\pdfoutput=1
\documentclass[leqno,a4paper,10pt]{article}

\usepackage[utf8x]{inputenc} \usepackage{a4wide} \usepackage{amsfonts}
\usepackage{amsmath} \usepackage{amssymb} \usepackage{amsthm}
\usepackage{mathrsfs} \usepackage{latexsym} \usepackage{graphicx}
\usepackage{bm} \usepackage{inputenc} \usepackage{enumitem}
\setlist[enumerate,1]{label=(\alph*), ref=(\alph*)}

\usepackage{url} \usepackage{hyperref}

\usepackage{graphicx}
\usepackage[dvipsnames]{xcolor}

\usepackage{mathtools} \mathtoolsset{showonlyrefs}

\swapnumbers

\newtheorem{Lemma}{Lemma}[section]
\newtheorem{Corollary}[Lemma]{Corollary}
\newtheorem{Theorem}[Lemma]{Theorem}

\theoremstyle{definition}

\newtheorem{Definition}[Lemma]{Definition}

\theoremstyle{remark}

\newtheorem{Remark}[Lemma]{Remark}

\newtheoremstyle{citing}
{3pt}
{3pt}
{\itshape}
{}
{\bfseries}
{.}
{.5em}
{\thmnote{#3}}
\theoremstyle{citing}

\newtheoremstyle{proof*}
{3pt}
{3pt}
{\rmfamily}
{}
{ \itshape}
{.}
{.5em}
{\thmnote{#3}}
\theoremstyle{proof*}
\newtheorem*{proof*}{}

\newcommand{\rhoF}[2]{\boldsymbol{\rho}^{#1}_{#2}}
\newcommand{\arho}[2]{\rhoF{#2}{#1}}
\newcommand{\arhol}[2]{\underline{\arho{#1}{#2}}}

\newcommand{\ar}[2]{\boldsymbol{r}_{#1}^{#2}}
\newcommand{\arl}[2]{\underline{\ar{#1}{#2}}}

\newcommand{\project}[1]{{#1}_\natural}

\newcommand{\perpproject}[1]{#1_\natural^\perp}

\newcommand{\measureball}[2]{{#1}\,{#2}}

\newcommand{\adim}{n}
\newcommand{\vdim}{d}

\newcommand{\npp}[2]{\boldsymbol{\xi}_{#2}^{#1}}
\newcommand{\anp}[2]{\npp{#2}{#1}}

\newcommand{\distF}[2]{\boldsymbol{\delta}_{#2}^{#1}}
\newcommand{\da}[2]{\distF{#2}{#1}}

\newcommand{\Unp}[1]{\mathrm{Unp}^{#1}}

\newcommand{\cballF}[3]{\mathbf{B}^{#1}(#2,#3)}

\newcommand{\oballF}[3]{\mathbf{U}^{#1}(#2,#3)}

\newcommand{\oball}[2]{\mathbf{U}(#1,#2)}

\newcommand{\cball}[2]{\mathbf{B}(#1,#2)}

\newcommand{\sphere}[1]{\mathbf{S}^{#1}}

\newcommand{\integers}{\mathbf{Z}}

\newcommand{\natp}{\integers_+}

\newcommand{\Q}{\mathbf{Q}}
\newcommand{\R}{\mathbf{R}}
\newcommand{\Real}[1]{\R^{#1}}

\newcommand{\Haus}[1]{ \mathscr{H}^{#1} }
\newcommand{\Leb}[1]{ \mathscr{L}^{#1}
}

\newcommand{\LM}{\mathscr{L}}

\newcommand{\HM}{\mathscr{H}}

\newcommand{\density}{\boldsymbol{\Theta}}

\newcommand{\unitmeasure}[1]{\boldsymbol{\alpha}(#1)}

\newcommand{\restrict}{ \mathop{ \rule[1pt]{.5pt}{6pt} \rule[1pt]{4pt}{0.5pt}}\nolimits }

\newcommand{\ud}{\ensuremath{\,\mathrm{d}}}

\newcommand{\uD}{\ensuremath{\mathrm{D}}}

\newcommand{\id}[1]{\mathbf{I}_{#1}}

\newcommand{\lIm}{[}
\newcommand{\rIm}{]}

\newcommand{\scale}[1]{\boldsymbol{\mu}_{#1}}
\newcommand{\trans}[1]{\boldsymbol{\tau}_{#1}}

\newcommand{\Clos}[1]{\mathop{\mathrm{Clos}}#1}

\DeclareMathOperator{\Hom}{Hom}

\newcommand{\Bdry}{\partial}

\DeclareMathOperator{\ap}{ap}

\DeclareMathOperator{\lin}{span}

\newcommand{\cnt}[1]{\mathscr{C}^{#1}}

\newcommand{\powerset}[1]{\mathbf{2}^{#1}}

\DeclareMathOperator{\Int}{Int}

\DeclareMathOperator{\dmn}{dmn}

\DeclareMathOperator{\grad}{grad}

\newcommand{\subgr}{\nabla}

\DeclareMathOperator{\Tan}{Tan}

\DeclareMathOperator{\Nor}{Nor}

\DeclareMathOperator{\without}{\sim}

\DeclareMathOperator{\im}{im}

 \DeclareMathOperator{\Cut}{Cut}

\newcommand{\Der}{\ensuremath{\mathrm{D}}}

\DeclareMathOperator{\pt}{pt}

\newcommand{\an}[2]{\bm{\nu}_{#1}^{#2}}

\title{Regularity of the distance function from arbitrary closed sets}
\author{S{\l}awomir Kolasi{\'n}ski \and Mario Santilli\footnote{Corresponding author.}}

\begin{document}

\maketitle

\begin{abstract}
    We investigate the distance function $\da{K}{\phi}$ from an arbitrary closed
    subset $ K $ of a~finite-dimensional Banach space $ (\R^{\adim}, \phi) $,
    equipped with a uniformly convex $ \cnt{2} $-norm $ \phi $. These spaces are
    known as \emph{Minkowski spaces} and they are one of the fundamental spaces
    of Finslerian geometry (see \cite{Martini2001}).  We prove that the gradient
    of~$\da{K}{\phi}$ satisfies a Lipschitz property on the complement of the
    $\phi$-cut-locus of~$K$ (a.k.a. the medial axis of $\R^{\adim} \without K$)
    and we prove a~structural result for the~set of~points outside~$K$ where
    $\da{K}{\phi}$ is pointwise twice differentiable, providing an answer to
    a~question raised by Hiriart-Urruty in~\cite{MR1539982}. Our~results give
    sharp generalisations of some classical results in the~theory of~distance
    functions and they are motivated by critical low-regularity examples for
    which the available results gives no~meaningful or very restricted
    informations.
    
    The results of this paper find natural applications in the theory of partial differential
    equations and in convex geometry. 
    
\end{abstract}

\section{Introduction}

For the basic notation we refer the reader to section~\ref{sec:notation}.

Suppose $ K \subseteq \Real{\adim} $ is a closed set and $\phi$ is a uniformly
convex norm on~$\Real{\adim}$; cf.~\ref{def:uniformly-convex-norm}. Our central
object of study is the \emph{$\phi$-distance function}
\begin{equation}
    \label{eq:def-delta}
    \da{K}{\phi}(x) = \inf \bigl\{ \phi(y-x) : y \in K \bigr\} \quad \text{for $x \in \Real{\adim}$} \,.
\end{equation}
We investigate in detail the set of points where $\da{K}{\phi}$ is not
differentiable and then also the set of points where it is not pointwise twice
differentiable.  Define
\begin{equation}
    \label{eq:def-Sigma}
    \Sigma^{\phi}(K) = (\R^{\adim} \without K) \cap
    \bigl\{
    x : \textrm{$\da{K}{\phi}$ is not differentiable at $x$}
    \bigr\} \,.
\end{equation}
A basic and fundamental result in the theory of distance functions asserts what follows.
\begin{Theorem}[$\mathcal{C}^{1,1}$-regularity]
    \label{Lions}
    If $ K \subseteq \R^{\adim} $ is an arbitrary closed set, then
    $ \da{K}{\phi} $ is $ \mathcal{C}^{1} $ with a locally Lipschitz gradient on
    the open subset
    $ U := \R^{\adim} \without \big(K \cup \Clos \Sigma^{\phi}(K)\big) $.
\end{Theorem}
\noindent This result can be deduced employing general results from the theory
of Hamilton-Jacobi equations (see~\cite[Theorem 15.1]{MR667669}
or~\cite{MR2060597}). Indeed, for a general closed set $ K $ it is well known
that $\da{K}{\phi} $ is a locally semiconcave function on $ \R^{\adim} \without
K $ and it satisfies, in a viscosity sense, the Eikonal equation $
\phi^\ast(\grad u) =1 $ on $ \R^{\adim} \without K $ (where $\phi^\ast$ is the
dual norm of $\phi $ as defined in~\ref{def:dual-norm}); see \cite{MR0310150},
\cite{MR667669}, \cite{MR699027}. For the Euclidean norm Theorem \ref{Lions} can
also be obtained using a purely geometric argument (see
\cite[4.8]{Federer1959}).

Of course, the conclusion of the theorem can be improved if we know that $ K $
is at least a $ \mathcal{C}^2 $-submanifold. In fact, in this case
$ \da{K}{\phi} $ is at least of class~$\cnt{2}$ on the open subset $ U $
and $\Clos \Sigma^{\phi}(K)$ is a set of $ \Leb{\adim} $-measure zero; if $ K $ is
a~$ \mathcal{C}^{2,1} $-submanifold, then $\Clos \Sigma^{\phi}(K)$ is a set of
locally finite $ \Haus{\adim-1} $-measure; see \cite{MR1695025}, \cite{MR1941909},
\cite{MR2094267}, \cite{MR2305073}, and \cite{MR2336304}. A~sufficient condition
that guarantees $ \Leb{\adim}(\Clos \Sigma(K)) =0 $ for closed
$ \mathcal{C}^{1,1} $-hypersurfaces $ K $ in terms of the inner radius of
curvature is given in \cite[Theorem 4.1]{MR3568029}. Moreover, if $ K $ is a
closed $ \mathcal{C}^1 $-hypersurface, then \cite[Theorem 1.3]{MR3568029}
provides a necessary and sufficient condition for a point
$ x \in \Real{\adim} \without K $ to lie in $ \Real{\adim} \without \Clos(\Sigma(K)) $.

On the other hand, it turns out that the $ \mathcal{C}^2 $-regularity is a
critical hypothesis; indeed the second named author has shown,
in~\cite{MR4279967}, that for a~convex open subset $ \Omega $ with $
\mathcal{C}^{1,1} $-boundary the set $ \Clos \Sigma^{\phi}(\R^{\adim} \without
\Omega) $ might have non empty interior in $ \Omega $; moreover, for a typical
(in the sense of Baire Category) convex open subset $ \Omega $ with $
\mathcal{C}^1 $-boundary we have that $\Sigma^{\phi}(\R^{\adim} \without \Omega)
$ is dense in $ \Omega $. There exist even closed $\mathcal{C}^{1, \alpha}
$-hypersurfaces $K$ such that $ \Sigma^{\phi}(K) $ is dense in all of $
\R^{\adim} $; see \cite[Corollary 2.9]{MR4279967}. In all these examples one can
choose $ \phi $ to be the Euclidean norm. Therefore, the set $ U $ defined
in~\ref{Lions} might easily be empty even if $ K = \Real{\adim} \without \Omega
$ and $ \Omega $ is a convex open subset with $ \mathcal{C}^{1} $ boundary, or
might reduce to a small tubular neighbourhood around $ K $ if $ \Omega $ has a $
\mathcal{C}^{1,1} $ boundary. Consequently Theorem~\ref{Lions} provides no (or
very limited) information in these situations. On the other hand it is well
known that the gradient of $ \da{K}{\phi} $ is a continuous map on its domain $
\Real{\adim} \without (K \cup \Sigma^{\phi}(K)) $. Therefore, it is a natural to
ask for a characterisation of the largest set on which the gradient of $
\da{K}{\phi} $ satisfies a~Lipschitz condition. We identify that set in
Theorem~\ref{intro_Lipschitz}, providing an effective sharp generalization of
Theorem~\ref{Lions} that is applicable in the aforementioned critical
low-regularity cases.
 
Besides its central role in Theorem \ref{Lions}, the set $ \Sigma^{\phi}(K) $
has been extensively studied in the last decades. Indeed, if we define the
\emph{$\phi $-nearest point projection} $\anp{K}{\phi}$ to be the multivalued
function (see~\ref{def:multivalued} and~\ref{def:multi-diff}) mapping a point $x
\in \Real{\adim}$ into the set
\begin{equation}
    \label{eq:def-xi}
    \anp{K}{\phi}(x) = K \cap \bigl\{a : \phi(x-a) = \da{K}{\phi}(x) \bigr\}, \,
\end{equation}
then it is well known that $ \Sigma^{\phi}(K) $ is precisely the set of points $
x \in \Real{\adim} \sim K $ where $\anp{K}{\phi}(x) $ is not a singleton. It
is remarkable that $ \Sigma^{\phi}(K) $ can be always covered by countably many
$ \mathcal{C}^2 $-hypersurfaces (see \cite{MR536060} and~\cite{Hajlasz2022});
moreover upper bounds on its Hausdorff measure are known (see
\cite{MR1185029}). Lower bounds and results on the propagation of the
non-differentiability points can be obtained from \cite{MR1760538},
\cite{MR1892229} and \cite{MR2538498}. The topological properties of the set $
\Sigma^{\phi}(K) $ in a~Euclidean or Riemannian setting are studied
in~\cite{MR1851184}, \cite{lieutier04}, \cite{MR3038120}.

Since $\da{K}{\phi}$ is locally semiconcave outside $ K $, it is a natural
question to investigate the set of points $ x \in \Real{\adim} \sim K $ where $
\da{K}{\phi} $ is pointwise twice differentiable, which means the set of points
where the function admits a second-order Taylor polynomial;
see~\ref{def:pt-diff}. Thus, we consider the set
\begin{equation}
    \label{eq:def-Sigma2}
    \Sigma^{\phi}_2(K) = (\R^{\adim} \without K)\cap \bigl\{
    x : \textrm{$ \da{K}{\phi} $ is not pointwise twice differentiable at $x$}
    \bigr\} \,.
\end{equation} 
A classical theorem on the twice differentiability of convex functions of
Alexandrov (see \cite{MR0003051}) readily implies the following result.
\begin{Theorem}\label{Alexandrov}
If $ K \subseteq \Real{\adim} $ is an arbitrary closed set, then $ \Leb{\adim}(\Sigma^{\phi}_2(K)) =0 $. 
\end{Theorem}
\noindent The example in \ref{convex_example} shows that the dimension of the
set $ \Sigma^{\phi}_2(K) $ might be exactly $ n $ even if $ K $ is a closed
convex body with $ \mathcal{C}^{1,1} $-boundary. On the other hand, it is
natural to ask about the structure of $ \Sigma^{\phi}_2(K) $ for a general
closed set $ K $; however, nothing is  known in the literature.
The~problem, in the~Euclidean setting, goes back at least to \cite{MR1539982} (see
last paragraph on page~458). We remark that the set of
twice-differentiability points of the $\phi$-distance function $\da{K}{\phi}$
corresponds to the set of differentiability points of the $ \phi $-nearest point
projection $\anp{K}{\phi}$; see
\ref{basic_properties_of_da_and_anp}\ref{basic_properties_of_da_and_anp_6}. Only
if $ K $ is convex sharp results are
available, that describe the structure of $ \Sigma^\phi_2(K) $ in terms of the \emph{unit $ \phi $-normal bundle} of~$K$. This is
defined for an arbitrary closed set $ K \subseteq \R^{\adim} $ as
\begin{equation}
    \label{eq:def-N-phi}
    N^{\phi}(K) =
    \bigl\{ (a, \eta)
    : a \in K, \; \eta \in \Real{\adim}, \;
    \phi(\eta) = 1, \;
    \da{K}{\phi}(a + s \eta) = s \;
    \textrm{for some $ s > 0 $}
    \bigr\} \,.
\end{equation} 
We recall that $ N^{\phi}(K) $ is a Borel and countably $(\adim-1)$-rectifiable
subset of~$\Real{2\adim}$; cf.~\cite[Lemma~5.2]{DRKS2020ARMA}, see
\cite[3.2.14(2)]{Federer1969} for the notion of rectifiability.
\begin{Theorem}
    \label{second order diff points convex} Suppose $ K \subseteq \R^{\adim} $ is
    convex. Then there exists $ Z \subseteq N^\phi(K) $ with $
    \Haus{\adim-1}(Z) =0 $ such that
    \begin{displaymath}
        \Sigma^{\phi}_2(K) = \{ a + r \eta : 0 < r < \infty, \; (a, \eta) \in Z
        \}.
    \end{displaymath}
    In particular, for $\Haus{\adim-1}$ almost all $(a, \eta) \in N^\phi(K) $
    the distance function $\da{K}{\phi} $ is pointwise twice differentiable at
    all points of the ray $ \{ a + r \eta: 0 < r < +\infty \} $.
\end{Theorem}
\noindent The exceptional set $Z$ cannot be excluded. In fact,
even if $ \phi $ is the Euclidean norm, there exist convex bodies $ K $ with
$ \mathcal{C}^{1,1} $ boundaries such that the set $ Z $ is dense in
$ N^{\phi}(K) $ with Hausdorff dimension $ \adim-1 $; see
\ref{convex_example}. Indeed, the construction of the
$ \mathcal{C}^{1,1} $-convex hypersurface in Theorem \ref{convex_example} shows
that one can choose $ Z $ to be somewhat arbitrarily complicated.  In the Euclidean setting
Theorem~\ref{second order diff points convex} is a~classical fact in convex
geometry; see~\cite{MR3155183}. The general anisotropic version in Theorem \ref{second order diff points convex} can be proved employing Theorem \ref{Lions} and following a similar argument. We also remark that for $ n =2 $ Theorem 1.3 can be deduced from a more general statement in \cite{MR3430855}. See also \cite{Hug2000-ju} for related results.

Our Theorem~\ref{intro:twice_diff_points} extends Theorem \ref{second order diff points convex} to arbitrary closed sets and it gives the~first answer to the question
of~Hiriart-Urruty, providing a new insight into the structure of $ \Sigma^{\phi}_2(K) $.

\subsection{The main results of the present paper}
In addition to the notions already introduced in the previous section, we
introduce here a few additional definitions and facts. Here $ K \subseteq
\Real{\adim} $ is always an arbitrary closed set. The~\emph{$\phi$-reach of~$K$}
is the function $\ar{K}{\phi}: N^{\phi}(K) \to (0, + \infty]$ given~by
\begin{equation}
    \label{eq:def-r-phi}
    \ar{K}{\phi}(a,\eta) = \sup \bigl\{ s >0 : \da{K}{\phi}(a + s\eta) = s \bigr\}
    \qquad \text{for $(a,\eta)\in N^{\phi}(K) $} \,.
\end{equation}
Simple arguments show that $ \ar{K}{\phi} $ is upper semicontinuous;
see~\ref{rem:rKphi-usc}. Moreover, we~define $\Cut^{\phi}(K)$, the
\emph{$ \phi $-cut locus of~$K$}, by
\begin{equation}
    \label{eq:def-Cut-locus}
    \Cut^{\phi}(K) = \bigl\{
    a + \ar{K}{\phi}(a, \eta)\eta : (a, \eta) \in N^{\phi}(K)
    \bigr\} \,.
\end{equation}
In view of this last definition, the number $\ar{K}{\phi}(a,\eta)$ can be seen
as \emph{the $ \phi $-distance of $a$ from the cut locus of $K$ in
  direction~$\eta$}; indeed, this function plays a central role in the seminal
work~\cite{MR1695025}, where it is proved to be Lipschitz continuous provided $
K $ is a smooth submanifold, and in other papers on the subject; see for
instance~\cite{MR2094267} and~\cite{MR2336304}). Note that (see
Remark~\ref{remark:inclusions})
\begin{equation*}
    \Sigma^{\phi}(K) \subseteq \Cut^{\phi}(K) \subseteq \Clos{\Sigma^{\phi}(K)}
    \,.
\end{equation*}
We notice also that $ \Leb{\adim}(\Cut^{\phi}(K)) = 0 $; see
Remark~\ref{Cut_Locus}. Since $ \Sigma^{\phi}(K) $ might not be nowhere dense,
the same is true for $ \Cut^\phi(K) $. Observe that $\anp{K}{\phi}$ induces a
natural fibration
\begin{displaymath}
    \Real{\adim} \without (\Cut^{\phi}(K) \cup K)
    = \bigl\{
    a + \rho\eta
    : (a, \eta) \in N^{\phi}(K) \,,
    0 < \rho < \ar{K}{\phi}(a, \eta)
    \bigr\}
    \,.
\end{displaymath}
Our goal is to study regularity properties of $\da{K}{\phi} $ on $\Real{\adim}
\without (\Cut^{\phi}(K) \cup K)$. To this end we look at the sets of points of
$ \Real{\adim} \without (\Cut^{\phi}(K) \cup K) $ with a \emph{uniform} positive
relative $ \phi $-distance to the cut-locus; in other words, we consider the
sets
\begin{equation}\label{eq: K sigma}
    K_{\sigma} = \bigl\{
    a + \rho \eta :
    (a, \eta) \in N^{\phi}(K)  ,\,
    0 < \sigma \rho \le \ar{K}{\phi}(a, \eta)
    \bigr\}
    \qquad \text{for $\sigma \ge 1$}  \,.
\end{equation}
Notice that $\Real{\adim} \without (\Cut^{\phi}(K) \cup K) = \bigcup_{\sigma >1}
K_{\sigma} $ but $ K_{\sigma} $ might have empty interior for every $\sigma >
1$.

Our first result asserts that $ \grad \da{K}{\phi} $ is locally Lipschitz
continuous on the sets $ K_{\sigma} $, which is a~sharp generalisation of
Theorem~\ref{Lions}. More precisely we prove the following result.
\begin{Theorem}[Lipschitz estimates for the gradient]
    \label{intro_Lipschitz}
    Suppose $\phi : \Real{\adim} \to \R$ is a uniformly convex norm of
    class~$\cnt{2}$ away from the origin, $K \subseteq \Real{\adim}$ is closed,
    $1 < \sigma < \infty$, $0 < s < t < \infty$, and
    \begin{displaymath}
        K_{\sigma,s,t} = \bigl\{ a + \rho \eta : (a, \eta) \in N^{\phi}(K) \,,
        s \leq \rho \leq t \,, \sigma \rho \le \ar{K}{\phi}(a, \eta) \bigr\} \,.
    \end{displaymath}
    Then $ \grad \da{K}{\phi} | K_{\sigma,s,t} $ is Lipschitz continuous.
\end{Theorem}
\noindent The restriction ``$\sigma \rho \le \ar{K}{\phi}(a, \eta)$'' cannot be
avoided since the~Lipschitz constant of $\grad \da{K}{\phi}$ may explode near
points of~$\Cut^{\phi}(K)$; cf.~\ref{rem:parabola}.  Observe that if $x \in
\Real{\adim} \without \bigl( \Clos{\Sigma^{\phi}(K)} \cup K \bigr)$, then $ x $
has positive distance from $\Cut^{\phi}(K) $; hence,
Theorem~\ref{intro_Lipschitz} includes Theorem \ref{Lions} as a~special case;
moreover, it is sharp in terms of specifying the set of points where a Lipschitz
condition for $\grad \da{K}{\phi}$ holds. There are two main difficulties in
proving \ref{intro_Lipschitz}. The first one arises from the fact that $
\Cut^{\phi}(K) $ might be dense in $ \Real{\adim} \without K $ and consequently
it does not seem to be possible to rely on general results for Hamilton-Jacobi
equations as for Theorem \ref{Lions}. The second difficulty comes from working
with a possibly non-Euclidean norm. In fact, if $ \phi $ is the Euclidean norm
then the proof of Theorem \ref{intro_Lipschitz} follows rather directly from the
geometric argument originally found by Federer for sets of positive reach in
\cite{Federer1959}, see \cite[3.10(1)]{MR4117503}. However, this argument is not
applicable if $ \phi $ is not the Euclidean norm, in which case one needs a
considerably more sophisticated approach, based on a careful analysis of the
geometric properties of the $ \phi$-balls (which occupies the entire section
\ref{sec:Lipschitz}). In fact this analysis allows to show that the $ \phi
$-nearest point projection $ \anp{K}{\phi} $ onto $ K $ satisfies the asserted
Lipschitz property. Recalling the well known relation between $ \anp{K}{\phi} $
and $ \grad \da{K}{\phi} $ (see
Lemma~\ref{basic_properties_of_da_and_anp}\ref{basic_properties_of_da_and_anp_3})
\begin{equation*}
    \grad \da{K}{\phi}(x) = \grad \phi(x - \anp{K}{\phi}(x)) \quad \textrm{for every $ x \in \dmn \grad \da{K}{\phi} $,}
\end{equation*}
we get the conclusion of~\ref{intro_Lipschitz}. Notice that \emph{uniform
  convexity and regularity} of the norm for $ \adim \geq 3 $ are crucial to
obtain the Lipschitz property in Theorem \ref{intro_Lipschitz}; see the last
section in~\cite{MR3430855}.

\paragraph{} The second goal of this paper is to extend Theorem \ref{second
  order diff points convex} to arbitrary closed sets. In case of convex sets the
reach function satisfies $ \ar{K}{\phi}(a, \eta) = + \infty $ for every $(a,
\eta)\in N^\phi(K) $ and consequently $ \Cut^\phi(K) = \varnothing $.  This is a
very special situation given by the assumption of convexity; indeed, even if we
consider $ \mathcal{C}^{1,1} $ convex hypersurfaces the reach function might be
discontinuous on a dense set and the cut-locus might not be nowhere dense;
see~\ref{rem:convex-hypersurface}. This suggests that a generalization of Theorem
\ref{second order diff points convex} to non-convex sets requires a careful analysis of the behaviour of~$ \ar{K}{\phi} $ and the connection with the points of differentiability of $ \anp{K}{\phi} $. This can be done considering the new reach-type function (recall~\eqref{eq: K sigma})
\begin{equation}
    \label{eq: arl}
    \arl{K}{\phi}(a, \eta)
    = \sup \bigl\{ \sigma r
    :  0 < r < \ar{K}{\phi}(a, \eta) ,\, \sigma > 1, \,
    \density^{\adim}(\Leb{\adim} \restrict K_{\sigma},a + r\eta) = 1 \bigr\} 
    \cup \{0\}
\end{equation}
for $(a, \eta)\in N^\phi(K)$. It holds that $0 \le \arl{K}{\phi}(a, \eta) \leq
\ar{K}{\phi}(a, \eta) $ for every $ (a, \eta)\in N^\phi(K)$; see Remark \ref{rem: arl less than ar}. However simple examples
show that there exist closed sets $ K $ for which $ \arl{K}{\phi}(a, \eta) <
\ar{K}{\phi}(a, \eta) $ for some $ (a, \eta) \in N^\phi(K) $;
cf.~\ref{example:Q-Cut_not_empty}.
If $ K $ is a closed set such that the $\phi
$-tubular neighbourhood $ \{ x : \da{K}{\phi}(x) < \rho \} $ of radius $ \rho >
0 $ does not intersect $ \Sigma^{\phi}(K) $, then $ \ar{K}{\phi}(a, \eta) \geq
\arl{K}{\phi}(a, \eta) \geq \rho $ for every $ (a, \eta)\in N^\phi(K) $;
cf.~Lemma \ref{remark:arl_positive_reach}.

Employing the reach function $ \arl{K}{\phi} $ we obtain the following result
on the structure of the $ \Sigma^{\phi}_2(K) $ for an arbitrary closed set $ K $.

\begin{Theorem}\label{intro:twice_diff_points}
    Suppose $K \subseteq \Real{\adim}$ is closed. Then 
    \begin{enumerate}
    \item
        \label{intro:twice_diff_points:1}
        $ \Cut^\phi(K) \subseteq \Sigma^{\phi}_2(K) $.
    \item
        \label{intro:twice_diff_points:propagation} If $(a, \eta)\in N^\phi(K) $ and there exists $ 0 < r < \arl{K}{\phi}(a, \eta) $ such that $ a + r\eta \notin \Sigma^\phi_2(K) $, then $ a + s\eta \notin \Sigma^\phi_2(K) $ for every $ 0 < s < \arl{K}{\phi}(a, \eta) $.
    \item
        \label{intro:twice_diff_points:2}
        $ \Haus{n-1} \bigl( \{ (a, \eta): \arl{K}{\phi}(a, \eta) < \ar{K}{\phi}(a,\eta)  \} \bigr) =0 $.
    \item
        \label{intro:twice_diff_points:3}
        there exist $ Z \subseteq N^\phi(K) $ with $ \Haus{\adim-1}(Z) =0 $ and a residual set
        \begin{equation*}
            R \subseteq \bigl\{ a + r\eta : \arl{K}{\phi}(a, \eta) \leq r < \ar{K}{\phi}(a, \eta) \bigr\}
        \end{equation*}
        such that
        \begin{equation*}
            \Sigma^{\phi}_2(K) \setminus \Cut^\phi(K)
            = \bigl\{ a + r\eta : 0 < r < \arl{K}{\phi}(a,\eta), \; (a, \eta) \in Z \bigr\} \cup R.
        \end{equation*}
    \end{enumerate}
    In particular, for $\Haus{\adim-1}$ almost all $(a, \eta) \in N^\phi(K) $ the
    distance function $\da{K}{\phi} $ is pointwise twice differentiable at all
    points of the line segment $ \{ a + r \eta: 0 < r < \ar{K}{\phi}(a, \eta) \} $.
\end{Theorem} 
\noindent We do not know whether $R = \{ a + r \eta: \arl{K}{\phi}(a, \eta) \leq
r < \ar{K}{\phi}(a, \eta) \}$; this is left as an open problem. In
Remark~\ref{remark : residualr set} we show, however, that
\begin{equation*}
    \Leb{n}( \{ a + r \eta: \arl{K}{\phi}(a, \eta)  \leq r < \ar{K}{\phi}(a, \eta)  \}   ) =0 \,.
\end{equation*}
The proof of Theorem \ref{intro:twice_diff_points} is based on the Lipschitz
property proved in Theorem \ref{intro_Lipschitz} and on some general estimates
for the pointwise principal curvatures of level sets of $\da{K}{\phi} $ (these
level sets might not even be topological manifolds but they admit a natural
notion of pointwise curvature; see~\ref{Gariepy-Pepe}). Moreover, results on the
preservation of the density points under bilipschitz transformations
(see~\cite{Buczo1992}) and on the approximate differentiability of multivalued
functions (see \ref{lem:lip+ap-diff=diff}) are used in a crucial way.

\subsection{Applications}

Here we briefly mention a couple of different applications of the results of the present paper.

\paragraph{Pointwise regularity and gradient Lipschitz estimates for solutions
  of the Eikonal equation.} Suppose $ \Omega \subseteq \Real{\adim} $ is an
arbitrary open set, $ \phi $ is a uniformly convex $ \mathcal{C}^2 $-norm and
$ \phi^\ast(u) = \sup\{ u \bullet v : \phi(v) =1 \} $. It is well known that
$ \da{K}{\phi} $ is the unique viscosity solution of the following Eikonal
equation on $ \Omega $
\[
\begin{cases}
    \phi^\ast(\grad u) =1 & \textrm{on $ \Omega $}\\
    u =0 & \textrm{on $ \partial \Omega $}
\end{cases}
\]
If $ \partial \Omega $ is hypersurface of class at least~$ \mathcal{C}^2 $, then
the local structure of this solution has been extensively studied and it is by
now very well understood (see the references cited at the beginning of this
introduction). On the other hand, as already explained, if $ \partial\Omega $ is
not $ \mathcal{C}^2 $ then such a solution can have a very complicated (in
particular dense!) singular set (see \cite{MR4279967}) and many classical
results in the theory do not give an insight about its local structure. In this
direction our results in Theorems~\ref{intro_Lipschitz}
and~\ref{intro:twice_diff_points} provide a new and rather sharp description of
the structure of the solutions of the Eikonal equation for \emph{arbitrary}
domains.

\paragraph{Steiner formula and curvature measures in uniformly convex finite
  dimensional Banach spaces} One of the original motivation of the second author
for the present work is to provide results that can be used to advance
the theory of convex and integral geometry in Minkowski spaces; see~\cite{Hug2000-ju}.
In~\cite{HugSan2022} the second author in collaboration with Daniel Hug
employs Theorems \ref{intro_Lipschitz} and \ref{intro:twice_diff_points} to
prove the Steiner formula for arbitrary closed sets in a uniformly convex Banach
space (Minkowski space); thus, extending the same formula previously obtained in
\cite{MR2031455} in the Euclidean space. The Steiner formula is then used as a
starting point to develop the theory of curvature measures for sets of positive
reach in a Minkowski space.

\section{Preliminaries}
\label{sec:prelim}

\subsection{Notation}
\label{sec:notation}

We follow traditional well established and widely accepted conventions and
notations typical for geometric measure theory. For convenience of the reader
we~briefly describe them here. We use the following symbols

\begin{enumerate}[itemsep=3pt,parsep=0pt,topsep=0.5em,align=left,itemindent=1em,labelsep=1em]
\item[$\R$] set of real numbers;
\item[$\overline{\R} = \R \cup \{ -\infty, +\infty \}$] extended reals;
\item[$\natp$] set of positive integers;
\item[$\varnothing$] the empty set;
\item[$\HM^\vdim$] the $\vdim$-dimensional Hausdorff measure;
\item[$\LM^\adim$] the Lebesgue measure over $\R^n$;
\item[$\unitmeasure{k}$] Lebesgue measure of the unit ball in $\R^k$;
\item[$\sphere{\adim-1}$] the unit Euclidean sphere in $\R^n$;
\item[$x \bullet y$] the inner product of two vectors $x$ and $y$ in a Euclidean space;
\item[$|x|$] the norm of a vector $x$ in a normed vectorspace;
\item[$A \without B$] set-theoretic difference of two sets $A$ and $B$;
\item[$\Clos{A}$] closure of a subset $A$ of a topological space;
\item[$\Int{A}$] interior of a subset $A$ of a topological space;
\item[$\Bdry{A} = \Clos{A} \without \Int{A}$] topological boundary of a~subset~$A$ of a topological space;
\item[$\dmn f$] domain of a function $f$;
\item[$f\lIm A \rIm$] the image of a set $A \subset \dmn f$ under the function $f$;
\item[$\im f$] the image of a function $f$, i.e., $\im f = f \lIm \dmn f \rIm$;
\item[$\uD f$] derivative of a function $f$ defined on a subset of a normed vectorspace; cf.~\ref{rem:pt-diff1};
\item[$\grad f$] gradient of a real-valued function $f$ defined on a subset of a Euclidean space;
\item[$A+B = \{ a+b : a \in A \,, b \in B \}$] algebraic sum of subsets $A$ and
    $B$ of a vectorspace;
\item[$\Hom(X,Y)$] vectorspace of linear maps of type $X \to Y$;
\item[$\Lambda x$ or $\langle x \,, \Lambda \rangle$] the value of a linear map
    $\Lambda$ on a vector $x \in \dmn \Lambda$;
\item[$\oballF{\phi}{x}{r} = \{ z : \phi(z-x) < r \}$] open ball with respect to a~norm~$\phi$;
\item[$\cballF{\phi}{x}{r} = \{ z : \phi(z-x) \le r \}$] closed ball with respect to a~norm~$\phi$;
\item[$f|A$] restriction of a function $f$ to the set $A \subseteq \dmn f$;
\item[$\subgr f(x)$] the set of subgradients of a~convex function~$f$
    at~$x \in \dmn f$; cf.~\ref{def:subgradient} and~\ref{rem:subgradients};
\item[$\project T$] the linear orthogonal projection onto a linear subspace $T$
    of a Euclidean space;
\item[$T^{\perp}$] the orthogonal complement of a linear subspace $T$ of a Euclidean space;
\item[\protect{$[ A \ni x \mapsto f(x)]$}] an unnamed function defined on $A$
    whose value at $x \in A$ is $f(x)$;
\item[$\mu \restrict A$] the restriction of a measure $\mu$ to a set $A$; cf.~\cite[2.1.2]{Federer1969};
\item[$\Tan(S,x)$] tangent cone at $x$ of a subset $S$ of a normed vectorspace; cf.~\cite[3.1.21]{Federer1969};
\item[$\Nor(S,x)$] normal cone at $x$ of a subset $S$ of a Euclidean space; cf.~\cite[3.1.21]{Federer1969};
\end{enumerate}

Given $k \in \natp$ and $0 < \alpha < 1$ we shall say that a function $f$ is
\emph{of class~$\cnt{k,\alpha}$} if the $k^{\mathrm{th}}$ derivative $\uD^kf$
exists and satisfies the H{\"o}lder condition with exponent~$\alpha$;
cf.~\cite[3.1.11 and 5.2.1]{Federer1969}. We~say that $f$ is \emph{of
  class~$\cnt{k}$} if $\uD^kf$ is just continuous.

\begin{Remark}
    We study several notions depending on the norm $\phi$, whose name is always
    in the superscript. In~case $\phi$ is the standard Euclidean norm
    on~$\R^{\adim}$ we omit it in the notation so, e.g., if $x \in \R^{\adim}$
    and $0 < r < \infty$, then $\oball xr$ denotes an open Euclidean ball in
    $\R^{\adim}$.
\end{Remark}

We~now introduce some classical functions

\begin{enumerate}[itemsep=3pt,parsep=0pt,topsep=0.5em,align=left,itemindent=*,labelsep=1em,label=$\bullet$]
\item Hausdorff densities of a Radon measure $\mu$ at $x$
    \begin{gather}
        \density^{\ast \adim}(\mu,x) = \limsup_{r \downarrow 0} \frac{\mu(\cball xr)}{\unitmeasure{\adim}r^{\adim}}  \,,
        \qquad
        \density^{\adim}_{\ast}(\mu,x) = \liminf_{r \downarrow 0} \frac{\mu(\cball xr)}{\unitmeasure{\adim}r^{\adim}}  \,,
        \\
        \text{and} \quad
        \density^{\adim}(\mu,x) = \density^{\ast \adim}(\mu,x)
        \quad \text{whenever $\density^{\ast \adim}(\mu,x) = \density^{\adim}_{\ast}(\mu,x)$} \,;
    \end{gather}
\item dilations
    \begin{displaymath}
        \scale{r}(x) = rx \quad \text{whenever $r \in \R$ and $x$ is a vector} \,;
    \end{displaymath}
\item translations
    \begin{displaymath}
        \trans{a}(b) = a + b \quad \text{whenever $a$ and $b$ are vectors in a vectorspace $X$} \,;
    \end{displaymath}
\item the identity map on a set $X$
    \begin{displaymath}
        \id{X}(x) = x \quad \text{whenever $x \in X$} \,.
    \end{displaymath}
\end{enumerate}

\begin{Remark}
    Without introducing any new symbols (in order not to make the notation too
    heavy) we find that given a function $f$ defined on a subset of a normed
    vectorspace
    \begin{enumerate}[itemsep=3pt,parsep=0pt,topsep=0.5em,align=left,itemindent=1em,labelsep=1em]
    \item[$\dmn \uD f$] is the set of differentiability points of $f$.
    \end{enumerate}
\end{Remark}

\begin{Remark}
    We shall repeatedly make use of the following simple fact. If $f$ is a~real
    valued function defined on a subset of a Euclidean space $X$,
    $x \in \dmn \uD^2 f$, and $u,v \in X$, then
    \begin{displaymath}
        \uD \grad f(x) u \bullet v = \langle u ,\, \uD \grad f (x) \rangle \bullet v = \uD^2f(x)(u,v) \,.
    \end{displaymath}
\end{Remark}

\begin{Remark}
    \label{rem:naming-convention}
    We adopt the convention that ``$C_{x.y}(a,b,c)$'' refers to the object
    (e.g. constant) defined in item (lemma, theorem, corollary, remark) x.y
    under the name~''$C$'', where $a$, $b$, $c$ should be substituted for
    parameters of x.y in order of their occurrence. For instance, if $v$ is a
    vector such that $\phi(v) = 1$, then
    $M_{\text{\ref{rem:graph}}}(\tfrac 12,v)$ is the manifold constructed by
    employing~\ref{rem:graph} with $\tfrac 12$ and $v$ in place
    of~``$\varepsilon$'' and~``$\eta$''.
\end{Remark}

\subsection{Basic concepts}

\begin{Definition}
    \label{def:sc-norm} We say that a norm $\phi : X \to \R$ is \emph{strictly
      convex} if for all $a,b \in X$
    \begin{displaymath}
        \phi(a+b) = \phi(a) + \phi(b) \quad \text{implies} \quad \phi(b) a =
        \phi(a) b \,.
    \end{displaymath}
\end{Definition}

\begin{Remark}
    In the sequel, unless otherwise specified, $\adim$ shall be a fixed positive
    integer, $X$ will be a~vectorspace of dimension $\adim$, and
    $\phi : X \to \R$ will be a~strictly convex norm on~$X$ of class~$\cnt{2}$
    away from the origin. Of~course, $X$ shall be isomorphic with~$\R^{\adim}$
    but, whenever we write $X$ instead of $\R^{\adim}$, we want to emphasise
    that there might not be a~natural choice of a~Euclidean structure on~$X$.
\end{Remark}

\begin{Definition}
    \label{def:dual-norm}
    Whenever $X$ is equipped with a scalar product and $\phi : X \to \R$ is a
    norm we define the \emph{conjugate norm} $\phi^{\ast} : X \to \R$ by the
    formula
    \begin{displaymath}
        \phi^{\ast}(x)
        = \sup \bigl\{
        x \bullet y : y \in X \,, \phi(y) = 1
        \bigr\}
        \quad \text{for $x \in X$} \,.
    \end{displaymath}
\end{Definition}

\begin{Definition}[\protect{cf.~\cite[2.12, 2.13]{DRKS2020ARMA}}]
    \label{def:uniformly-convex-norm} Assume $X$ is equipped with a Euclidean
    structure. We say that $\phi : X \to \R$ is a~\emph{uniformly convex norm}
    if it is a norm and there exists $\gamma > 0$ such that the function $\bigl[
    X \ni x \mapsto \phi(x) - \gamma |x| \bigr]$ is convex.
\end{Definition}

\begin{Remark}[\protect{cf.~\cite[2.32]{DRKS2020ARMA}}]
    \label{rem:phi-phiast} If $\phi$ is a uniformly convex norm of
    class~$\cnt{2}$ away from the origin, then $\phi^{\ast}$ is also a uniformly
    convex norm of class~$\cnt{2}$ away from the origin. Moreover, $\grad
    \phi^{\ast}|S^*$ is the inverse of $\grad \phi|S$, where $S = \Bdry
    \cballF{\phi}{0}{1}$ and $S^* = \Bdry \cballF{\phi^{\ast}}{0}{1}$.
\end{Remark}

\begin{Definition}
    Given a closed set $ K \subseteq X $ we define
    \begin{displaymath}
        S^{\phi}(K,r) = \bigl\{ x : \da{K}{\phi}(x) = r \bigr\}
        \quad \text{for $r > 0$} \,.
    \end{displaymath}
\end{Definition}

\begin{Definition}
    \label{def:multivalued}
    A map of the type $f : X \to \powerset{Y}$ shall be called
    \emph{$Y$-multivalued}. In case $x \in X$ and $f(x)$ is a~singleton,
    we~abuse the notation and write $f(x)$ to denote the unique member
    of~$f(x)$.
\end{Definition}

\begin{Definition}\label{def: restrictions}
    Let $ f $ be a $ Y $-multivalued function on $ X $ and $ A \subseteq X
    $. Then we denote with $ f|A $ the $ Y $-multivalued map on $ X $ defined as
    \begin{equation*}
        (f|A)(x) = f(x) \quad \textrm{if $ x \in A $}, \qquad (f|A)(x)= \varnothing \quad \textrm{if $ x \notin A $.}
    \end{equation*}
\end{Definition}

\begin{Definition}\label{def: inverse}
    Let $ f $ be a $ Y $-multivalued function on $ X $ and $ A \subseteq X
    $. Then we define the inverse $ f^{-1} $ of $ f $ as the $ X $-multivalued
    map on $ Y $ as
    \begin{equation*}
        f^{-1}(y) = \{x : y \in f(x)\} \qquad \textrm{for $ y \in Y $.}
    \end{equation*}
\end{Definition}

\begin{Definition}
    \label{anisotropic npp}
    Suppose $K \subseteq X$ is closed and $\anp{K}{\phi} : X \to \powerset{K}$
    is the $\phi$-nearest point projection onto $K$ characterised
    by~\eqref{eq:def-xi}.
    \emph{The Cahn-Hoffman map of $K$ associated to $\phi$} is the multivalued
    map~$\an{K}{\phi} : X \without K \to \powerset{\Bdry \cballF{\phi}{0}{1}}$ defined by the formula
    \begin{displaymath}
        \an{K}{\phi}(x) = \da{K}{\phi}(x)^{-1}(x - \anp{K}{\phi}(x)) \quad
        \text{for $x \in X \without K$} \,.
    \end{displaymath}
\end{Definition}

\begin{Remark}
    \label{closedness of anp(x)}
    It will be useful to notice that $\anp{K}{\phi}(x)$ is a compact subset
    of~$X$ for every $x \in X$.
\end{Remark}

\begin{Remark}
    Since $\phi$ is a norm, one readily checks that \emph{if $ a \in K $, $ v
      \in X $ and $ \da{K}{\phi}(a+v) = \phi(v) $, then $ \da{K}{\phi}(a+tv) = t
      \phi(v) $ for every $ 0 \leq t \leq 1 $.}
\end{Remark}

\begin{Remark}
    \label{single-valuedness_of_anp} It has been observed in
    \cite[2.38(g)]{DRKS2020ARMA}, using strict convexity of~$\phi$, that
    \emph{if $ a \in K $, $ u \in \Bdry \oballF{\phi}{0}{1} $,
      $ 0 < t < \infty $ and $ \da{K}{\phi}(a+tu) = t $, then
      $ \anp{K}{\phi}(a+su) $ is a singleton and $\anp{K}{\phi}(a+su) = \{a\}$
      for every $ 0 < s <t $.}
\end{Remark}

\begin{Definition}[\protect{cf.\ \cite[p.\ 213]{MR0274683}}]
    Let $f : X \to \overline{\R}$ and $x,v \in X$. The \emph{one-sided
      directional derivative of $f$ at $x$ with respect to $v$} is defined to be
    \begin{displaymath}
        f'(x;v) = \lim_{\lambda \to 0+}\frac{f(x+\lambda v) - f(x)}{\lambda} \,,
    \end{displaymath} whenever the limit exists in $\overline{\R}$.
\end{Definition}

\begin{Remark}
    If $f$ is a convex function and $x$ is a~point with $f(x) \in \R$, then $
    f'(x;v) $ exists for every $ v \in X $; cf.~\cite[Theorem 23.1]{MR0274683}.
\end{Remark}

\begin{Definition}[\protect{cf.~\cite[p.\ 214-215 and Theorem 23.2]{MR0274683}}]
    \label{def:subgradient}
    Suppose $f : X \to \overline{\R}$ is convex and $x \in X$ is such that
    $f(x) \in \R$. We say that $ \zeta \in X $ is a \emph{subgradient of $f$ at
      $x$} if
    \begin{displaymath}
        f'(x;v) \geq \zeta \bullet v \quad \text{for $ v \in X $.}
    \end{displaymath}
    The set of all subgradients of $f$ at $x$ is denoted by $\subgr f(x)$.
\end{Definition}

\begin{Remark}
    \label{rem:subgradients}
    Since the symbol ``$\Bdry{}$'' is used in this paper for the topological
    boundary of a set and, on grounds of set theory, functions are sets it would
    introduce ambiguities if we used the standard notation ``$\partial f$'' for
    the subgradient mapping of~$f$; hence, we decided to denote it
    ``$\subgr f$''.
\end{Remark}

In the next definition we use the notion of a~\emph{polynomial function} which
is formally defined in~\cite[1.10.4]{Federer1969}.

\begin{Definition}
    \label{def:pt-diff}
    Let $X$, $Y$ be normed vectorspaces and $f$ be a function mapping a subset
    of $X$ into $Y$. We say that \emph{$f$ is pointwise differentiable of order
      $k$ at $x$} if there exist: an open set $U \subseteq X$ such that $x \in U
    \subseteq \dmn f$ and a polynomial function $P : X \to Y$ of degree at most
    $k$ such that $f(x) = P(x)$ and
    \begin{displaymath}
        \lim_{y \to x}\frac{|f(y) - P(y)|}{|y-x|^k} = 0 \,.
    \end{displaymath} Whenever this holds $P$ is unique and the \emph{pointwise
      differential of
      order $i$ of $f$ at $x$}, for $i = 1, \ldots ,k$, is defined by $\pt\Der^i
    f(x) = \Der^i P(x) $. As usual $\pt \Der^1 f(x) = \pt\Der f(x)$.
\end{Definition}

\begin{Remark}
    \label{rem:pt-diff1}
    The notion of pointwise differentiability of order $1$ coincides with the
    classical notion of differentiability so $\pt\Der = \Der$;
    cf.~\cite[3.1]{Federer1969}. A summary of known facts about pointwise
    differentiability for functions can be found, e.g.,
    in~\cite[\S{2}]{Menne2019}.
\end{Remark}

\begin{Remark}
    \label{non emptyness of the subdifferential} If $f$ is a $\R$-valued convex
    function on an open subset $U$ of $X$ then $\subgr f(x)$ is non empty for
    every $x \in U$; cf.~\cite[Theorem~23.4]{MR0274683}. Moreover, $f$ is
    differentiable of order~$1$ at~$x$ if and only if $\subgr f(x)$ is a
    singleton; cf.~\cite[25.1]{MR0274683}.
\end{Remark}

We need to extend the concept of continuity and differentiability to multivalued
maps.

\begin{Definition}[\protect{cf.~\cite[Definition~2]{Zajicek1983diff}}]
    \label{def:multi-cont}
    Let $X$ and $Y$ be normed vectorspaces and $T$ be a~$Y$-mul\-ti\-val\-ued
    map defined on $X$. We say that \emph{$T$ is weakly continuous
      at $a \in X$} if and only if $T(a) \neq \varnothing $ and for each $\varepsilon > 0$
    there exists $\delta > 0$ such that
    \begin{displaymath}
        T(x) \subseteq T(a) + \oball{0}{\varepsilon} \quad \text{whenever $x \in
          $ and $|x-a| < \delta$} \,.
    \end{displaymath} If, additionally, $T(x)$ is a singleton, then we say that
    $T$ is \emph{continuous at $x$}.
\end{Definition}

\begin{Remark}
  We notice that if $T(y) = \varnothing$ for $y \in \cball{x}{\delta}
    \without \{x\}$ then $T$ is continuous at~$x$. On the other hand, we remark that studying the
    map $\anp{K}{\phi}$ we do not need to worry about such strange behaviour.
    Moreover, in~\ref{basic_properties_of_da_and_anp}\ref{basic_properties_of_da_and_anp_5}
    we prove that $\anp{K}{\phi}$ is weakly continuous on the whole of
    $\mathbf{R}^n$. Obviously, $\anp{K}{\phi}(x)$ is a~singleton for all $x \in X$ if and
    only if $K$ is convex.
\end{Remark}

\begin{Remark}
    Note that weakly continuous multivalued functions may carry connected sets
    into disconnected sets. Consider, e.g., the function $f : \R \to
    \powerset{\R}$ given by $f(t) = \{-1\}$ if $t < 0$, $f(t) = \{ 1 \}$ if $t >
    0$, and $f(0) = \{ -1 \,, 0 \,, 1 \}$; then, $f$ is weakly continuous in the
    sense of~\ref{def:multi-cont}. Another example is $\anp{K}{\phi}$ which is
    weakly continuous on the whole of~$\Real{\adim}$ regardless of the choice of
    the closed set~$K \subseteq \Real{\adim}$; in~particular, when $K$ is
    disconnected;
    cf.~\ref{basic_properties_of_da_and_anp}\ref{basic_properties_of_da_and_anp_5}.
\end{Remark}

\begin{Definition}[\protect{cf.~\cite[Definition~3]{Zajicek1983diff}}]
    \label{def:multi-diff}
    Let $X$, $Y$ be finite dimensional normed vectorspaces and $T$ be
    a~$Y$-multivalued map defined on $X$. We~say that \emph{$T$ is
      differentiable at $a \in X$} if and only if $T(a)$ is a singleton and
    there exists a~linear map $L : X \to Y$ such that for any $\varepsilon > 0$
    there exists $\delta > 0$ satisfying
    \begin{displaymath}
        | y - T(a) - L(x-a)| \leq \varepsilon |x-a|
        \quad \text{whenever $|x-a| \le \delta$ and $y \in T(x)$} \,.
    \end{displaymath} The set of all such $L$ is denoted by~$\Der T(a)$. In case
    $\Der T(a)$ is a~singleton, we say that $T$ is \emph{strongly differentiable
      at~$a$}.
\end{Definition}

\begin{Remark}
    \label{rem:multi-diff} Note that it might happen that $T(y) = \varnothing$
    for some $y \in \cball{x}{\delta}$. Actually, if $T(x) \ne \varnothing$ and
    there exists $\delta > 0$ such that $T(y) = \varnothing$ for $y \in
    \cball{x}{\delta} \without \{x\}$, then $T$ is differentiable at~$x$ with
    $\Der T(x) = \Hom(X,Y)$. On the other hand if, e.g., $\dim X = \adim$ and
    $\density^{\adim}(\Leb{\adim} \restrict \{ y : T(y) = \varnothing \}, x) =
    0$, then $\Der T(x)$ is a singleton.
\end{Remark}

\begin{Remark}\label{composition of multivalued functions}
    Let $ P $ and $ Q $ be two multivalued functions and $ x \in \Real{\adim}
    $. If $ P $ is differentiable at $ x $ and $ Q $ is differentiable at $ P(x)
    $ then the multivalued function $ R $ given by
    \begin{displaymath}
        R(y) = Q[P(y)] = {\textstyle \bigcup}\{ Q(w) : w \in P(y) \} \quad
        \text{for $ y \in \Real{\adim} $} \,,
    \end{displaymath} is differentiable at $x$.
\end{Remark}

\begin{Definition}\label{def: arho}
    \label{anisotropic_rho}
    Let $ K \subseteq \Real{\adim} $ be closed. For $x \in \Real{\adim}$ define $ \arho{K}{\phi}: \mathbf{R}^n \rightarrow \overline{\mathbf{R}} \cap \{t : 1 \leq t \leq \infty\} $ as
    \begin{displaymath}
        \arho{K}{\phi}(x)
        = \sup \R \cap \bigl\{ s : \da{K}{\phi}(a + s(x-a)) = s \da{K}{\phi}(x) \bigr\} 
        \end{displaymath}
    whenever $ x \in \mathbf{R}^n $ and $ a \in \anp{K}{\phi}(x) $.
\end{Definition}

\begin{Remark}\label{rho well definition}
    Definition~\ref{anisotropic_rho} is well posed,
    since~\ref{single-valuedness_of_anp} gives that if $ \anp{K}{\phi}(x) $ is
    not a~singleton, then
    \begin{displaymath}
        \sup \bigl\{ s : \da{K}{\phi}(a + s(x-a)) = s \da{K}{\phi}(x) \bigr\} = 1
        \quad \text{for every $ a \in \anp{K}{\phi}(x)$} \,.
    \end{displaymath}
\end{Remark}

\noindent The following Lemma will be used in section \ref{sec:twice-diff}.

\begin{Lemma}\label{rem:rho-usc}
    For every closed set $ K \subseteq \mathbf{R}^n $ the function $
    \arho{K}{\phi} $ is upper semicontinuous and satisfies
    \begin{equation*}
        \arho{K}{\phi}(x) = t \arho{K}{\phi}(a + t (x-a))
        \qquad \textrm{for $ x \in \mathbf{R}^{n+1} $, $ a \in \anp{K}{\phi}(x) $ and $ 0 < t \leq \arho{K}{\phi}(x) $} \,.
    \end{equation*}
    Moreover, $ \Cut^\phi(K) = \mathbf{R}^n \cap \{x : \arho{K}{\phi}(x) = 1\} $.
\end{Lemma}

\begin{proof}
    Let $x_0, x_1, x_2, \ldots \in \R^{\adim}$ and $\beta \in \R$ be such that
    $\lim_{i \to \infty} x_i = x_0$, $\phi(x_i - x_0) < 1$ for $i \in \natp$, and
    $\lim_{i \to \infty} \arho{K}{\phi}(x_i) > \beta$. Since $\da{K}{\phi}$ is
    continuous we have
    $\lim_{i \to \infty} \da{K}{\phi}(x_i) = \da{K}{\phi}(x_0)$ and we may
    assume $\da{K}{\phi}(x_i) < \da{K}{\phi}(x_0) + 1$ for $i \in \natp$.
    Choose $a_i \in \anp{K}{\phi}(x_i)$ for $i \in \natp$. Since
    $\{a_i : i \in \natp \} \subseteq \cballF{\phi}{x_0}{\da{K}{\phi}(x_0) + 2}$
    we may, possibly choosing a subsequence, assume that
    $\lim_{i \to \infty} a_i = a_0$ and then $a_0 \in \anp{K}{\phi}(x_0)$ by
    continuity of both~$\da{K}{\phi}$ and~$\phi$. Assume further that
    $\arho{K}{\phi}(x_i) \ge \beta$ for $i \in \natp$. Recalling the definition
    of $\arho{K}{\phi}$ we~obtain
    \begin{displaymath}
        \da{K}{\phi}(a_0+\beta(x_0-a_0))
        = \lim_{i \to \infty} \da{K}{\phi}(a_i+\beta(x_i-a_i))
        = \lim_{i \to \infty} \beta \da{K}{\phi}(x_i)
        = \beta \da{K}{\phi}(x_0) \,;
    \end{displaymath}
    hence, $\arho{K}{\phi}(x_0) \ge \beta$. Since this holds for any $\beta \in
    \R$ satisfying $\lim_{i \to \infty} \arho{K}{\phi}(x_i) > \beta$, we see
    that $\lim_{i \to \infty} \arho{K}{\phi}(x_i) \le \arho{K}{\phi}(x_0)$ and
    we conclude that $ \arho{K}{\phi} $ is upper semicontinuous.
    
    Suppose $ x \in \mathbf{R}^{n+1} $, $ a \in \anp{K}{\phi}(x) $ and $ 0 < t
    \leq \arho{K}{\phi}(x) $ and we prove that $\arho{K}{\phi}(x) = t
    \arho{K}{\phi}(a + t (x-a)) $. Evidently, if $\arho{K}{\phi}(x) = \infty$,
    then $\arho{K}{\phi}(a + t (x-a)) = \infty$ for all $0 < t < \infty$ and the
    assertion~is~true. Therefore, we assume $1 \le \arho{K}{\phi}(x) < \infty$
    and define $ y = a + t(x-a) $. Notice that $ \da{K}{\phi}(y) = t
    \da{K}{\phi}(x) $ and $ a \in \anp{K}{\phi}(y) $. Since
    \begin{gather}
        a + \frac{\arho{K}{\phi}(x)}{t}(y-a) = a + \arho{K}{\phi}(x)(x-a) \\
        \text{we have} \quad
        \da{K}{\phi} \big( a + \tfrac{\arho{K}{\phi}(x)}{t}(y-a) \big)
        = \arho{K}{\phi}(x)\da{K}{\phi}(x)
        = \tfrac{\arho{K}{\phi}(x)}{t}\da{K}{\phi}(y) \,; \\
        \label{rem:rho-usc eq}        
        \text{hence,} \quad
        t \arho{K}{\phi}(y) \geq \arho{K}{\phi}(x) \,.
    \end{gather}
    Noting that $x = a + \frac{1}{t}(y-a)$, $ \arho{K}{\phi}(y) \geq \frac{1}{t}
    $, and $ a \in \anp{K}{\phi}(y) $, we can apply the inequality in
    \eqref{rem:rho-usc eq}, replacing $ x$ and $ t $ with $ y $ and $
    \frac{1}{t} $ respectively, to obtain the reverse inequality; hence, equality.
    
    Finally the assertion about the cut locus follows directly from the
    definition of $ \arho{K}{\phi} $.
\end{proof}

\begin{Lemma}\label{lem: dilation homeomorphism}
    Suppose $ K \subseteq \R^n $ is closed, $ \sigma > 1 $, $ K_\sigma = \{ x :
    \arho{K}{\phi}(x) \geq \sigma\} \sim K $ and the $ \mathbf{R}^n
    $-multivalued function $ h_t $ is defined as
    \begin{gather}
        h_t(y) = ty + (1-t) \anp{K}{\phi}(y) \quad \text{for $y \in \Real{\adim}$ and $ t \in \mathbf{R} $}.
    \end{gather}
    Then the map $ h_t|K_\sigma $ is a homeomorphism onto $ K_{\sigma/t} $ with
    $ (h_t|K_\sigma)^{-1} = h_{1/t}| K_{\sigma/t} $ for every $ 0 < t < \sigma$.
\end{Lemma}

\begin{proof}
    Since $ t \arho{K}{\phi}(h_t(x)) = \arho{K}{\phi}(x) \geq \sigma $ for every
    $ x \in K_\sigma $ by Lemma~\ref{rem:rho-usc}, we get that $ h_t[K_\sigma]
    \subseteq K_{\sigma/t} $. Let $ y \in K_{\sigma/t} $ and define $ x =
    h_{1/t}(y)$. Since $\anp{K}{\phi}(y)$ is a singleton we can write $x =
    \anp{K}{\phi}(y) + \frac{1}{t}(y - \anp{K}{\phi}(y)) $. Notice that $
    \anp{K}{\phi}(x) = \anp{K}{\phi}(y) $ and $\frac{1}{t}\arho{K}{\phi}(x) =
    \arho{K}{\phi}(y) \geq \frac{\sigma}{t} $, again by
    Lemma~\ref{rem:rho-usc}. We conclude that $ x \in K_\sigma $ and, by a
    direct computation, $ h_t(x) = y $. It follows that $ h_t \circ h_{1/t} =
    \mathbf{I}_{K_{\sigma/t}} $ and $ h_t[K_\sigma] = K_{\sigma/t} $.
    
    Since $ 0 < \frac{1}{t} < \frac{\sigma}{t} $ we apply the statement proved
    in the last paragraph with $ t $ and $ \sigma $ replaced by $ \frac{1}{t} $
    and $ \frac{\sigma}{t} $ respectively to infer that $ h_{1/t} \circ h_{t} =
    \mathbf{I}_{K_\sigma} $ and $ h_{1/t}[K_{\sigma/t}] = K_\sigma $. This
    proves that $ h_t|K_\sigma $ is an homeomorphism onto $ K_{\sigma/t} $.
\end{proof}

The next lemma provides an alternative description of the normal
bundle~$N^\phi(K)$ defined in~\eqref{eq:def-N-phi} and the reach function
defined in~\eqref{eq:def-r-phi}.
\begin{Lemma}
    \label{rem:rKphi-usc}
    For every closed set $ K \subseteq \mathbf{R}^n $ the function
    $\ar{K}{\phi}: N^\phi(K) \rightarrow \mathbf{R} \cap \{t : 0 < t \leq
    \infty\}$ is upper semicontinuous. Moreover,
      \begin{equation*}
          N^\phi(K) = \bigl\{ (\anp{K}{\phi}(x), \an{K}{\phi}(x)): x \in \mathbf{R}^n \sim (K \cup \Cut^\phi(K)) \bigr\}
      \end{equation*}
      and 
      \begin{equation*}
          \ar{K}{\phi}(\anp{K}{\phi}(x), \an{K}{\phi}(x))
          = \da{K}{\phi}(x) \arho{K}{\phi}(x)
          \qquad \textrm{for all $ x \in \Real{n} \sim (K \cup \Cut^\phi(K)) $} \,.
      \end{equation*}
\end{Lemma}

\begin{proof}
    Assume this is not true, so that for each $i
    \in \natp$ there is $(a_i,u_i) \in N^{\phi}(K)$ such that
    \begin{gather}
        \lim_{i \to \infty} a_i = a \in K \,,
        \quad
        \lim_{i \to \infty} u_i = u \in \sphere{\adim-1} \,,
        \quad
        (a,u) \in N^\phi(K) \,,
        \\
        \text{and} \quad
        \ar{K}{\phi}(a,u) < \lim_{i \to \infty} \ar{K}{\phi}(a_i,u_i) \,.
    \end{gather}
    Let $s \in \Real{}$ be such that
    \begin{displaymath}
        0 < \ar{K}{\phi}(a,u) < s \le \ar{K}{\phi}(a_i,u_i) \quad \text{for $i \in \natp$} \,.
    \end{displaymath}
    Since $\ar{K}{\phi}(a,u) < s$ we can find $b \in K$ such that $\phi((a+su) - b) <
    s$. Let $\varepsilon \in \Real{}$ be such that
    \begin{displaymath}
        0 < \varepsilon < s 
        \quad \text{and} \quad
        0 < \phi((a+su) - b) < s - \varepsilon \,.
    \end{displaymath}
    Let $i \in \natp$ be so big that $\phi(a_i - a) \le 2^{-3} \varepsilon$ and
    $\phi(u_i - u) \le 2^{-3} s^{-1} \varepsilon$. Then
    \begin{displaymath}
        \phi( (a_i + su_i) - (a + su) ) \le \phi(a_i - a) + s \phi(u_i - u) \le 2^{-2} \varepsilon \,.
    \end{displaymath}
    Since $\ar{K}{\phi}(a_i,u_i) \ge s$ we get a contradiction
    \begin{multline}
        s = \da{K}{\phi}(a_i + su_i)
        \le \phi( (a_i + su_i) - b )
        \\
        \le \phi( (a_i + su_i) - (a + su) ) + \phi( (a + su) - b )e
        \le 2^{-2} \varepsilon + s - \varepsilon < s \,.
    \end{multline}
    The second part of the statement follows mechanically from the definitions.
\end{proof}

\begin{Remark}
    Notice that in~\cite[Remark~5.6]{DRKS2020ARMA} we erroneously claim that
    $\ar{K}{\phi}$ is lower semicontinuous which is obviousy wrong but,
    fortunatelly, does not affect other results of~\cite{DRKS2020ARMA} since we
    only need the fact that $\ar{K}{\phi}$ is a Borel function there.
\end{Remark}

\begin{Remark}
    \label{rem:convex-hypersurface}
    The function $ \ar{K}{\phi} $ can fail to be continuous even if $ K $ is
    a~compact convex $ \mathcal{C}^{1,1} $ hypersurface.  In~fact
    in~\cite{MR4279967} we show that there exists a compact and convex
    $ \mathcal{C}^{1,1} $-hypersurface $ K $ such that $ \Clos(\Sigma(K)) $ has
    non empty interior. Noting that $ N(K) $ is the classical unit normal bundle
    of $ K $ and consequently it is compact, we~infer that if $\ar{K}{}$ was
    continuous then $ \Cut(K) $ would be compact; consequently
    $ \Clos(\Sigma(K)) = \Cut(K) $ and $ \Leb{\adim}(\Cut(K)) > 0 $ which is
    incompatible with Remark~\ref{Cut_Locus}.
\end{Remark}

\subsection{Auxiliary results}

The following lemma shows that if $A \subseteq \Real{\adim}$ is a~set of points
at which a~multivalued function~$f$ satisfies a~Lipschitz condition, $a$~is a
density points of~$A$, and $f|A$ is differentiable at~$a$, then $f$ is
differentiable at~$a$. It is a variant of a classical result stating that
a~Lipschitz function that is approximately differentiable at a~point is
classically differentiable at that point; cf.~\cite[3.1.5]{Federer1969}.

\begin{Lemma}
    \label{lem:lip+ap-diff=diff} Assume
    \begin{gather}
        a \in A \subseteq \Real{\adim} \,, \quad C \in \R \,, \quad f :
        \Real{\adim} \to \powerset{\Real{\adim}} \,, \quad
        \density^{\adim}(\Leb{\adim} \restrict (\Real{\adim} \without A), a) = 0
        \,, \\ \text{$f(b)$ is a singleton for $b \in A$} \,, \quad \text{$f|A$
          is differentiable at~$a$} \,, \\
        \label{eq:adl:lip}
        | f(b) - y | \le C | b - c |
        \quad \text{whenever $b \in A$, $c \in \Real{\adim}$, $y \in f(c)$} \,.
    \end{gather} Then $f$ is strongly differentiable at~$a$.
\end{Lemma}

\begin{proof}
    Since $a$ is a density point of~$A$ we see that $f|A$ is strongly
    differentiable at~$a$ and $\uD f(a) = \{L\}$ for some $L \in
    \Hom(\Real{\adim},\Real{\adim})$; cf.~\ref{rem:multi-diff}.
    Let~$\varepsilon > 0$. Choose $0 < \delta < \varepsilon$ such that
    \begin{gather}
        \Leb{\adim}(\cball ar \without A) < \varepsilon^{\adim} 4^{-\adim}
        (\|L\|+C)^{-\adim} \unitmeasure{\adim} r^{\adim} \quad \text{whenever $0
          < r \le 2\delta$} \,, \\
        | f(b) - f(a) - L(b-a) | < \tfrac 12 \varepsilon |b-a|
        \quad \text{whenever $b \in A \cap \cball{a}{2\delta}$} \,.
    \end{gather}
    Let $c \in \cball{a}{\delta}$ and $y \in f(c)$. Set $r = |c-a|$ and choose
    $b \in A$ such that $|c-b| = \da{A}{}(c) \le r$. Clearly
    $\cball{c}{|c-b|} \subseteq \cball{a}{2r}$ and
    $\measureball{\Leb{\adim}}{\cball{c}{|c-b|}} = \unitmeasure{\adim} |c-b|^{\adim}$;
    hence,
    \begin{displaymath}
        (\|L\|+C) |c-b| \le \tfrac 12 \varepsilon | c - a | \,.
    \end{displaymath} Since $b \in A$ we obtain
    \begin{multline}
        | y - f(a) - L(c-a) |
        \le | y - f(b) | + | f(b) - f(a) - L(b-a) | + | L (b - c) | \\ \le C | c
        - b | + \tfrac 12 \varepsilon | b - a | + \| L \| \cdot | c - b | \le
        \varepsilon |c-a| \,.  \qedhere
    \end{multline}
\end{proof}

The next lemma is a classical result in convex analysis.

\begin{Lemma}\label{Alexandrov theorem}
    If $ U \subseteq \Real{\adim} $ is an open convex set, $ f : U \to \Real{} $
    is a convex function and $ x \in U $, then the following three statements
    are equivalent.
    \begin{enumerate}
    \item
        \label{Alexandrov theorem:1}
        $f$ is pointwise differentiable of order~$2$ at~$x$.
    \item
        \label{Alexandrov theorem:2}
        The multivalued map $\subgr f$ is differentiable at $x$.
    \item
        \label{Alexandrov theorem:3}
        There is at least one function $g : U \to \Real{\adim}$ such that $g(y)
        \in \subgr f(y)$ for every $y \in U$ and $ g $ is differentiable at~$x$.
    \end{enumerate}
    If~\ref{Alexandrov theorem:1}, \ref{Alexandrov theorem:2},
    and~\ref{Alexandrov theorem:3} hold, then
    \begin{displaymath}
        \Der \subgr f(x)u \bullet v
        = \Der g(x)u \bullet v
        = \pt\Der^{2}f(x)(u,v)
        \quad \text{for $ u,v \in \Real{\adim} $} \,.
    \end{displaymath}
\end{Lemma}

\begin{proof}
    Clearly $\subgr f(y) \ne \varnothing$ for all $y \in U$ because $f$ is
    convex and~\ref{non emptyness of the subdifferential}. The proof that
    \ref{Alexandrov theorem:1} implies \ref{Alexandrov theorem:3} is contained
    in \cite[p. 495]{MR556619} (and attributed to Fitzpatrick). For the proof
    that \ref{Alexandrov theorem:3} implies \ref{Alexandrov theorem:2} and
    \ref{Alexandrov theorem:2} implies \ref{Alexandrov theorem:1}, one can look
    in \cite{MR534228}. In fact, first we notice that $ f $ is "zweimal
    differenzierbar in~$p$" in the sense of \cite[4.2]{MR534228} if and only if $
    \subgr f $ is differentiable at $ p $ in the sense of \ref{def:multi-diff};
    then we look at \cite[4.3]{MR534228} and \cite[4.8]{MR534228} respectively.
\end{proof}

\begin{Definition}
    \label{def:semiconcave} Suppose $ U \subseteq \Real{\adim}$ is open. We say
    that a function $ g : U \to \Real{} $ is \emph{semiconcave} if and only if
    there exists $\kappa \geq 0$ such that the function $g(y) - (\kappa/2)
    |y|^{2}$ is concave.
\end{Definition}

The following lemma collects few facts on the continuity, differentiability, and
convexity properties of $\da{K}{\phi}$ and $ \anp{K}{\phi} $ for an arbitrary
closed set~$K$.

\begin{Lemma}
    \label{basic_properties_of_da_and_anp}
    Let $ K \subseteq \Real{\adim} $ be a closed set. Then the following
    statements hold.
    \begin{enumerate}
    \item
        \label{basic_properties_of_da_and_anp_1} $(\da{K}{\phi})'(x;v) =
        \inf\ \bigl\{ \grad \phi(x-y) \bullet v : y \in \anp{K}{\phi}(x) \bigr\} $ for every
        $ v \in \Real{\adim} $ and $ x \in \Real{\adim} \without K $.
    \item
        \label{basic_properties_of_da_and_anp_2} For each $ x \in \Real{\adim}
        \without K $ there exists an open neighbourhood $ U \subseteq
        \Real{\adim} \without K $ of $ x $ such that $ \da{K}{\phi}|U $ is
        semiconcave.
    \item
        \label{basic_properties_of_da_and_anp_3} $ \da{K}{\phi} $ is
        differentiable at $ x \in \Real{\adim} \without K $ if and only if $
        \anp{K}{\phi}(x) $ is a singleton, in which case
        \begin{displaymath}
            \grad \da{K}{\phi}(x) = \grad \phi(x - \anp{K}{\phi}(x)) \,, \quad
            \anp{K}{\phi}(x) = x - \da{K}{\phi}(x)\grad \phi^{\ast}(\grad
            \da{K}{\phi}(x)) \,.
        \end{displaymath}
    \item
        \label{basic_properties_of_da_and_anp_4} If $ \da{K}{\phi} $ is
        differentiable at $ x \in \Real{\adim} \without K $ then $ \da{K}{\phi}
        $ is differentiable at $ \anp{K}{\phi}(x) + t(x-\anp{K}{\phi}(x)) $ for
        $ 0 < t < \arho{K}{\phi}(x) $ with
        \begin{displaymath}
            \grad \da{K}{\phi}(x) = \grad \da{K}{\phi} \bigl( \anp{K}{\phi}(x) +
            t(x-\anp{K}{\phi}(x)) \bigr) \,.
        \end{displaymath}
    \item
        \label{basic_properties_of_da_and_anp_6} $\da{K}{\phi}$ is pointwise
        differentiable of order $2$ at $x \in \Real{\adim} \without K$ if and
        only if $\anp{K}{\phi}$ is differentiable at~$x$ in the sense
        of~\ref{def:multi-diff}, in which case
        \begin{displaymath}
            \pt\Der^2\da{K}{\phi}(x)(u,v) = \Der(\grad \phi \circ
            \an{K}{\phi})(x)(u) \bullet v \quad \text{for $ u, v \in
              \Real{\adim} $} \,.
        \end{displaymath}
    \item
        \label{basic_properties_of_da_and_anp_5} $\anp{K}{\phi}$ is weakly
        continuous in the sense of~\ref{def:multi-cont}.
    \end{enumerate}
\end{Lemma}

\begin{proof}
    The assertions \ref{basic_properties_of_da_and_anp_1} and
    \ref{basic_properties_of_da_and_anp_2} correspond to \cite[Corollary to
    Theorem 3*]{MR699027} and \cite[Theorem 5]{MR699027}, respectively. 

    We prove \ref{basic_properties_of_da_and_anp_3}. If $ \anp{K}{\phi}(x) $ is
    a~singleton, then for every $ v \in \Real{\adim} $ the partial derivative of
    $\da{K}{\phi}$ at $x$ with respect to $v$ exists and equals $\grad
    \phi(x-\anp{K}{\phi}(x)) \bullet v$
    by~\ref{basic_properties_of_da_and_anp_1}. Since $\da{K}{\phi}$ is Lipschitz
    continuous with Lipschitz constant~$1$ by~\cite[Lemma~2.38(a)]{DRKS2020ARMA}
    and $(\da{K}{\phi})'(x;\an{K}{\phi}(x)) = 1$
    by~\cite[Lemma~2.32(c)]{DRKS2020ARMA} we conclude that $\da{K}{\phi}$ is
    differentiable at~$x$ using~\cite[2.4, 2.5]{Fitzpatrick1984}. On the other
    hand if $\anp{K}{\phi}(x)$ is not a~singleton then $ \da{K}{\phi} $ is not
    differentiable at $ x $ by a~result of Konjagin~\cite{MR0493113} (see also
    \cite[Proposition~2]{MR699027}).

    Assertion \ref{basic_properties_of_da_and_anp_4} follows from
    \ref{basic_properties_of_da_and_anp_3} and \ref{single-valuedness_of_anp}.
    
    To prove \ref{basic_properties_of_da_and_anp_6} we observe that for $ x \in
    \Real{\adim} \without K $ there exist,
    by~\ref{basic_properties_of_da_and_anp_2}, a~constant $ \kappa > 0 $,
    an~open neighbourhood $ U $ of $ x $, and a~convex function $ V : U \to
    \Real{} $ such that
    \begin{displaymath}
        V(y) = (\kappa/2)|y|^{2} - \da{K}{\phi}(y) \quad \text{for $ y \in U $.}
    \end{displaymath} Moreover, we observe, using
    \ref{basic_properties_of_da_and_anp_1}, that if $ \xi : U \to \Real{\adim} $
    is a function such that $ \xi(y) \in \anp{K}{\phi}(y) $ for every $ y \in U
    $, then $ \kappa y-\grad \phi(y-\xi(y)) = \kappa y - \grad
    \phi\big(\da{K}{\phi}(y)^{-1}(y - \xi(y))\big) \in \subgr V(y) $. Therefore,
    we conclude from~\ref{Alexandrov
      theorem} that $ \da{K}{\phi} $ is pointwise differentiable of order~$ 2 $
    at~$ x $ if and only if~$ \anp{K}{\phi} $ is differentiable at~$ x $. The
    displayed equation in~\ref{basic_properties_of_da_and_anp_6} also follows
    from the postscript of~\ref{Alexandrov theorem}.

    Finally we prove \ref{basic_properties_of_da_and_anp_5}. The argument used
    in \cite[2.38(b)]{DRKS2020ARMA}, which proves the statement for the
    restriction of $ \anp{K}{\phi} $ to the set of points where it is
    single-valued, also works in the general case of
    \ref{basic_properties_of_da_and_anp_5}. For completeness we provide a
    proof. By contradiction we assume there are $ x \in \Real{\adim} $, $
    \varepsilon > 0 $ and two sequences $x_{i}\in \Real{\adim} $ and $a_{i}\in
    K$ such that $ x_{i} \to x $, $ a_{i} \in \anp{K}{\phi}(x_{i}) $ and $
    |a_{i} - b| \geq \varepsilon $ for every $ b \in \anp{K}{\phi}(x) $ and for
    every $ i \geq 1 $. Noting that
    \begin{displaymath}
        |\da{K}{\phi}(x_{i})-\da{K}{\phi}(x)| \leq \phi(x_{i}-x)
    \end{displaymath} and
    \begin{displaymath}
        \phi(a_{i}-x) \leq \da{K}{\phi}(x_{i}) + \phi(x_{i}-x) \leq
        \da{K}{\phi}(x) + 2F(x_{i}-x)
    \end{displaymath} for every $ i \geq 1 $, it follows that $\{a_{i}: i \geq 1
    \}$ is a bounded sequence and consequently we can assume $ a_{i} \to a $ for
    some $ a \in K $. Then
    \begin{displaymath}
        \da{K}{\phi}(x) = \lim_{i \to \infty}\da{K}{\phi}(x_{i}) = \lim_{i \to
          \infty}\phi(a_{i}-x_{i}) = \phi(x-a), \quad a \in \anp{K}{\phi}(x) \,.
    \end{displaymath} It follows that $|a_{i} - a| \geq \varepsilon $ for every
    $ i \geq 1 $, which is in contradiction with $ a_{i} \to a $.
\end{proof}

\begin{Remark}
    Continuity properties of $\anp{K}{\phi}|U$ will be studied more carefully
    in~\ref{lem:xi-uni-cont} in case~$\phi$ is strictly convex and
    in~\ref{thm:lipxi-minkowski} in case $\phi$ is uniformly convex.
\end{Remark}

\begin{Lemma}
    \label{Gariepy-Pepe aux}
    Assume $T$ is an hyperplanes in $\Real{\adim}$, $ \alpha \in T $,
    $ f : T \to T^{\perp} $ is function continuous at $ \alpha $,
    $ a = \alpha + f(\alpha) $, $ A = \{ \chi + f(\chi): \chi \in T \} $ and
    $ \Tan(A,a) \subseteq T $. Then $f$ is differentiable at $ \alpha $,
    $\uD f(\alpha) = 0$, and $ \Tan(A,a) = T $.
\end{Lemma}

\begin{proof}
    We prove that $\uD f(\alpha)$ exists and equals zero. If
    $\limsup_{T \ni \chi \to \alpha} |f(\chi) - f(\alpha)| \cdot |\chi -
    \alpha|^{-1} > 0$, then we could find a sequence $\chi_j \in T$ such that
    $\chi_j \to \alpha$,
    $(\chi_j - \alpha) \cdot |\chi_j - \alpha|^{-1} \to w \in T$, and
    $(f(\chi_j) - f(\alpha)) \cdot |\chi_j - \alpha|^{-1} \to v \in T^{\perp}$
    with $v \ne 0$ as $j \to \infty$; hence, setting $v_j = \chi_j + f(\chi_j)$
    we would obtain $(v_j - a) \cdot |v_j - a|^{-1} \to w \in \Tan(A,a)$ as
    $j \to \infty$ and $\perpproject{T}w \ne 0$ which would contradict
    $ \Tan(A,a) \subseteq T $ by the definition of tangent cone;
    cf.~\cite[3.1.21]{Federer1969}.
\end{proof}

The following Lemma follows rather directly from classcial implicit function
theorems for Lipschitz and semiconcave functions. In the next lemma, given $ x
\in \Real{\adim} $, $ \epsilon, \delta > 0 $, and a linear space $T \subseteq
\Real{\adim}$, we make use of cylinders aligned to $T$ defined the following way
\begin{equation*}
    U_{\epsilon, \delta}(x,T)
    = \bigl\{
    y : |\project{T}(y-x) |< \delta,\; |\project{T}^\perp(y-x)| < \epsilon
    \bigr\}
    \,.
\end{equation*}
Moreover, we recall from Lemma \ref{basic_properties_of_da_and_anp} that
\begin{equation*}
    \grad \da{K}{\phi}(x) \neq 0
    \qquad \textrm{for every $ x \in (\dmn \grad \da{K}{\phi}) \sim K $} \,.
\end{equation*}
  
\begin{Lemma}
    \label{Gariepy-Pepe}
    Suppose $ K $ is a closed subset of $ \Real{\adim+1} $, $ r > 0 $,
    $ x \in S^{\phi}(K,r) $, $ \da{K}{\phi} $ is differentiable at $ x $,
    $ \nu = \grad \da{K}{\phi}(x)/|\grad \da{K}{\phi}(x)| $ and
    $ T = \Real{\adim} \cap \{ v : v \bullet \nu =0 \} $.
    
    Then $ T = \Tan(S^{\phi}(K,r),x) $ and there are $ \epsilon, \delta > 0 $
    and a semiconcave function $ f : T \to \R $ such that $ f $ is
    differentiable at $ \project{T}x $ with $ \Der f(\project{T}x) =0 $,
    \begin{equation}\label{Gariepy-Pepe eq1}
        U_{\epsilon, \delta}(x,T) \cap S^{\phi}(K,r) = U_{\epsilon, \delta}(x, T) \cap \{ \chi + f(\chi)\nu : \chi \in T \}
    \end{equation} 
    and 
    \begin{equation}\label{Gariepy-Pepe eq2}
        U_{\epsilon, \delta}(x, T) \cap \{x : \da{K}{\phi}(x) \geq r   \} = U_{\epsilon, \delta}(x, T) \cap \{ \chi + t\nu: t \leq f(\chi)   \}.
    \end{equation}
    Moreover, if $ \da{K}{\phi} $ is pointwise differentiable of order $ 2 $ at
    $ x $ then $ f $ is pointwise differentiable of order $ 2 $ at
    $ \project{T}x $ and
    \begin{displaymath}
        |\grad \da{K}{\phi}(x)|\,   \pt\Der^{2} f(\project{T}x)(u,v)  =
        -\pt\Der^{2}\da{K}{\phi}(x)(u,v) \quad \text{for $ u,v \in T $.}
    \end{displaymath}
\end{Lemma}

\begin{proof}
    We notice that $\da{K}{\phi} $ is locally semiconcave on $ \Real{\adim} \sim
    K $ by Lemma
    \ref{basic_properties_of_da_and_anp}\ref{basic_properties_of_da_and_anp_2}. Since
    $ \da{K}{\phi} $ is differentiable at $ x $ and $ \grad \da{K}{\phi}(x) \neq
    0 $, noting Remark \ref{non emptyness of the subdifferential} and
    \cite[Remark~1.4]{MR816398}, we see that we can apply
    \cite[Theorem~3.3]{MR816398} to find $ \epsilon, \delta > 0 $ and a
    semiconcave function $ f : T \rightarrow \R $ such that \eqref{Gariepy-Pepe
      eq1} and \eqref{Gariepy-Pepe eq2} hold\footnote{At a first sight we can
      only deduce from \cite[Theorem 3.3]{MR816398} that there exist $ \epsilon,
      \delta > 0 $, an hyperplane $ S \subseteq \Real{\adim} $ and a semiconcave
      function $ f : S \rightarrow \R $ such that \eqref{Gariepy-Pepe eq1} and
      \eqref{Gariepy-Pepe eq2} with $ S $ replaced by $ T $. However, closer
      inspection of the proof of~\cite[Theorem~3.3]{MR816398} reveals that we
      can choose $ S = T $, as the existence of a lipschitzian function $ f : T
      \rightarrow \R $ which satisfies \eqref{Gariepy-Pepe eq1} for some $
      \epsilon, \delta > 0 $ directly follows from Clarke implicit function
      theorem.}. Since $ \da{K}{\phi} $ is differentiable at $ x $, then $
    \Tan(S^\phi(K,r), x) \subseteq T $. Therefore, the first part of the
    conclusion follows from Lemma \ref{Gariepy-Pepe aux}.

    Assume now that $ \da{K}{\phi} $ is pointwise differentiable of order $ 2 $
    at $ x $ and $ x =0 $. Setting $\zeta = \chi + f(\chi)\nu$, we notice that
    $\uD \da{K}{\phi}(0)(\zeta) = f(\chi) |\grad \da{K}{\phi}(0)|$ so
    \begin{align}
        0 &= \lim_{T \ni \chi\to 0} \frac{\da{K}{\phi}(\zeta)
          - \da{K}{\phi}(0)-  f(\chi)\,| \grad\da{K}{\phi}(0)|
          - \frac 12 \pt\Der^2\da{K}{\phi}(0)(\zeta, \zeta)}{|\chi|^2}
        \\
        &= -\lim_{T \ni \chi \to 0} \frac{ f(\chi)\,| \grad\da{K}{\phi}(0)|
          + \frac 12 \pt\Der^2\da{K}{\phi}(0)(\chi,\chi)}{|\chi|^2},
    \end{align}
    which means that $ f $ is pointwise differentiable of order $ 2 $ at $ 0 $
    with
    \begin{displaymath}
       |\grad \da{K}{\phi}(0)|\, \pt\Der^2f(0) 
        = -\pt\Der^2\da{K}{\phi}(0)| T \times T \,.
        \qedhere
    \end{displaymath}
\end{proof}

\begin{Lemma}
    \label{lem:existence_of_real_eigenvalues:aux}
    Suppose $T$ is a hyperplane in $\Real{\adim}$, $ f : T \to T^\perp $ is a
    function of class~$\cnt{2}$ such that $ f(0)=0 $ and $ \Der f(0) =0 $,
    $ \Sigma = \{\chi + f(\chi): \chi \in T\} $, and
    $ \eta : \Sigma \to \sphere{\adim-1} $ is a function of class~$\cnt{1}$ such
    that $ \eta(x) \in \Nor(\Sigma,x)$ for $ x \in \Sigma $. Then
    \begin{displaymath}
        \Der \eta(0) u \bullet v = - \Der^2f(0)(u,v) \bullet \eta(0) \quad
        \text{for $u,v \in T$} \,.
    \end{displaymath}
\end{Lemma}

\begin{proof}
    Noting that $ \eta(\chi + f(\chi)) \bullet(u + \Der f(\chi)u) =0 $ for $ u
    \in T $ and $ \chi \in T $, we differentiate this relation with respect to $
    \chi $ at $ 0 $.
\end{proof}

\section{Lipschitz estimates}
\label{sec:Lipschitz}
In this section we consider an abstract Minkowski space $(X,\phi)$ of
dimension~$\adim$ and \emph{we are defining a~Euclidean structure on~$X$} to fit
our problem. For this reason we choose to denote the space with ``$X$'' rather
than ``$\Real{\adim}$'' since the latter refers to a space with a predefined
Euclidean structure which is of no use to us. The operator norm of o bilinear
map $\Lambda : X \times X \to X$ with respect to~$\phi$ is defined as
in~\cite[1.10.5]{Federer1969}, i.e.,
\begin{displaymath}
    \| \Lambda \|_\phi = \sup \bigl\{ \phi(\Lambda(x,y)) :
    x,y \in X ,\, \phi(x) \le 1 ,\, \phi(y) \le 1 \bigr\} \,.
\end{displaymath}
Once the Euclidean structure on $X$ is defined we shall use the symbol
$\|\Lambda\|$ to denote the operator norm of $\Lambda$ with respect to that
Euclidean structure.

\begin{Definition}[\protect{cf.~\cite[4.1]{Federer1959}}]
    \label{def:unp-npp}
    Let $K \subseteq X$ be closed. We define the set of points with \emph{unique
      nearest point}
    \begin{displaymath}
        \Unp{\phi}(K) = X \cap \bigl\{
        x : \Haus{0}(\anp{K}{\phi}(x)) = 1
        \bigr\} \,.
    \end{displaymath}
\end{Definition}

We start by showing that $\anp{K}{\phi}$ is uniformly continuous on certain
sets. Later, in~\ref{thm:lipxi-minkowski} and~\ref{cor:xi-lip-cont}, we
bootstrap this regularity to Lipschitz continuity. Uniform continuity is
obtained for strictly convex norms $\phi$, while Lipschitz continuity requires
uniform convexity and $\cnt{2}$~regularity of~$\phi$.

\begin{Lemma}
    \label{lem:xi-uni-cont} Assume
    \begin{gather}
        \text{$\phi$ is strictly convex} \,, \quad K \subseteq X \text{ is
          closed} \,, \quad 1 < \lambda < \infty \,, \\ K_{\lambda} = (X \without
        K) \cap \{ x : \arho{K}{\phi}(x) \ge \lambda \} \,.
    \end{gather}
    Then there exists $\omega_{\lambda} : \R \to \R$ such that
    $\lim_{t \downarrow 0} \omega_{\lambda}(t) = 0$ and
    \begin{displaymath}
        \phi( a - b ) \le \da{K}{\phi}(x) \, \omega_{\lambda} \bigl(
        \phi(x-y) / \da{K}{\phi}(x) \bigr) \quad \text{for $x \in
          K_{\lambda}$, $y \in X$, $a \in \anp{K}{\phi}(x)$, $b \in
          \anp{K}{\phi}(y)$} \,.
    \end{displaymath}
\end{Lemma}

\begin{proof}
    For $0 \le t < \infty$ define
    \begin{gather}
        K_{\lambda}(t) = X \times X \cap \bigl\{ (a,b) : \phi(a) = \lambda \,,
        \phi(b) \ge \lambda \,, \phi \bigl( (1-1/\lambda) a - b \bigr) \le 1 +
        2t \bigr\} \,, \\ \omega_{\lambda}(t) = \sup \bigl\{ \phi(a-b) : (a,b)
        \in K_{\lambda}(t)) \bigr\} \,.
    \end{gather}
    Observe that strict convexity of $\phi$ yields
    \begin{displaymath}
        {\textstyle \bigcap} \bigl\{ K_{\lambda}(t) : 0 < t < \infty \bigr\} = X
        \times X \cap \bigl\{ (a,a) : \phi(a) = \lambda \bigr\} \quad \text{and}
        \quad \lim_{t \downarrow 0} \omega_{\lambda}(t) = 0 \,.
    \end{displaymath}
    Indeed, assume $\limsup_{t \downarrow 0}
    \omega_{\lambda}(t) = \delta$.  Find sequences $X \cap \{ a_j : j \in \natp
    \}$ and $X \cap \{ b_j : j \in \natp \}$ such that $(a_j,b_j) \in
    K_{\lambda}(1/j)$, $\phi(a_j - b_j) \ge \delta - 1/j$, $\lim_{j \to \infty}
    a_j = a_0$ and $\lim_{j \to
      \infty} b_j = b_0$ with $\phi(a_0) = \lambda$, $\phi(b_0) \ge \lambda$,
    $\phi(b_0 - a_a) \ge \delta$, $\phi(z_0 - (1 - 1/\lambda)a_0) \le 1$. Then
    \begin{displaymath}
        \lambda \le \phi(b_0) \le \phi(b_0 - (1 - 1/\lambda)a_0) + \phi((1 -
        1/\lambda)a_0) \le 1 + \lambda - 1 = \lambda
    \end{displaymath}
    which implies that $a_0 = b_0$ and $\delta = 0$ by~\ref{def:sc-norm}.

    Let $x \in K_{\lambda} \subseteq \Unp{\phi}(K)$, $y \in X$. Choose
    \begin{gather}
        \bar a \in \anp{K}{\phi}(x) \,, \quad \bar b \in \anp{K}{\phi}(y) \,,
        \quad c = \bar a + \lambda (x - \bar a) \,, \\ r = \da{K}{\phi}(x)
        \,, \quad a = (\bar a - c)/r \,, \quad b = (\bar b - c)/r \,, \quad t =
        \phi(x-y)/r \,.
    \end{gather}
    Clearly we have
    \begin{align}
      \phi(\bar b-x) &\le \phi(\bar b-y) + \phi(y-x)
      \\
                     &\le \phi(\bar a-y) + \phi(y-x)
                       \le \phi(\bar a-x) + 2 \phi(x-y)
                       = r (1 + 2t) \,.
    \end{align}
    Since $(x-c)/r = (1 - 1/\lambda)a$ we obtain
    \begin{displaymath}
        r \phi( (1 - 1/\lambda) a - b ) = \phi((x-c) - (\bar b-c)) = \phi(x -
        \bar b) \le r (1+2t) \,.
    \end{displaymath} Because $x \in K_{\lambda}$ we know also that $\phi(\bar a
    - c) < \phi(\bar b - c)$; hence,
    \begin{displaymath}
        r \phi(b) = \phi(\bar b - c) > \phi(\bar a - c) = r \phi(a) = \lambda
        \phi(x - \bar a) = \lambda r \,.
    \end{displaymath} This shows that $(a,b) \in K_{\lambda}(t)$ so $\phi(a-b)
    \le \omega_{\lambda}(t)$ and $\phi(\bar a - \bar b) \le r
    \omega_{\lambda}(t)$.
\end{proof}

\begin{Corollary}
    \label{cor:xi-uni-cont} Assume $\phi$ is strictly convex, $K \subseteq X$ is
    closed, $0 < s < t < \infty$, $1 < \lambda < \infty$, and
    \begin{displaymath}
        K_{\lambda,s,t} = \bigl\{ x : \arho{K}{\phi}(x) \ge \lambda \,, s \le
        \da{K}{\phi}(x) \le t \bigr\} \,.
    \end{displaymath} Then $\anp{K}{\phi}|K_{\lambda,s,t}$ is uniformly
    continuous.
\end{Corollary}

\begin{Remark}
    This provides an alternative proof that $\anp{K}{\phi} | \Unp{\phi}(K)$ is
    continuous; cf.~\cite[2.42]{DRKS2020ARMA}.
\end{Remark}

\begin{Remark}
    \label{mr:setup}
    Assume that~$X$ is a~finite dimensional vectorspace equipped with a~strictly
    convex and continuously differentiable (away from the origin) norm
    $\phi : X \to \R$. We~define
    \begin{gather}
        \label{eq:setup}
        S = \Bdry{\cballF{\phi}01} \,, \quad
        \xi : X \without \{0\} \to S
        \quad \text{by} \quad \xi(x) = x \phi(x)^{-1}
        \quad \text{for $x \in X \without \{0\}$} \,,
        \\
        \pi : S \to \Hom(X,X) \quad \text{by} \quad \pi = \uD \xi | S
        \,.
    \end{gather} Note that whenever $\eta \in S$ the map $\pi(\eta)$ is a
    projection onto $\Tan(S,\eta)$ such that
    \begin{equation}
        \label{eq:pi-props} \pi(\eta) \circ \pi(\eta) = \pi(\eta) \,, \quad \im
        \pi(\eta) = \Tan(S,a) \,, \quad \eta \in \ker \pi(\eta) \quad \text{for
          $\eta \in S$} \,.
    \end{equation}
\end{Remark}

\begin{Lemma}
    \label{lem:pi-injective}
    Consider the situation as in~\ref{mr:setup}. Let $0 < \varepsilon < 1$ and
    set
    \begin{displaymath}
        R = \sup \R \cap \bigl\{ r : 0 < r < 1 \,, \eta,\zeta \in S \,,
        \phi(\eta-\zeta) \le r \text{ implies } \| \pi(\eta) - \pi(\zeta)
        \|_\phi \le 1-\varepsilon \bigr\} \,.
    \end{displaymath}
    Then $\pi(\eta) | S \cap \cballF \phi{\eta}R$ is injective whenever
    $\eta \in S$.
\end{Lemma}

\begin{proof}
    Assume that for some $\eta \in S$ the map $\pi(\eta)|S \cap \cballF
    \phi{\eta}R$ is not injective. Set
    \begin{displaymath}
        D = S \cap \cballF \phi{\eta}R
    \end{displaymath}
    and let $\xi,\zeta \in D$ be such that $\pi(\eta) \xi = \pi(\eta) \zeta$;
    hence, $\xi-\zeta \in \ker \pi(\eta) = \lin \{\eta\}$. Assume
    $\phi(\xi-\eta) \le \phi(\zeta-\eta)$. If $\eta = \xi$, then $\zeta = -\eta$
    and $\phi(\zeta-\eta) = 2 > 1$ which cannot happen because $\zeta \in D$ and
    $R \le 1$. Let $P = \lin \{ \eta \,, \xi \}$. Then
    $\zeta = \xi + \lambda \eta$ for some $\lambda \in \R$ and we get
    \begin{displaymath}
        \eta,\xi,\zeta \in P \cap S \,.
    \end{displaymath} Let $\gamma : \R \to S \cap P$ be such that
    \begin{displaymath}
        \phi(\gamma'(t)) > 0 \quad \text{for $t \in \R$}\,, \quad \gamma(0) =
        \xi \,, \quad \gamma(1) = \zeta \,.
    \end{displaymath} Set $A = \im \gamma|[0,1]$. Since $\xi - \zeta \in \lin
    \{\eta\}$ we see that both $\xi$ and $\eta$ are on the same side of the line
    $\lin\{\eta\}$ in~$P$. Therefore, the Monotonicity
    Lemma~\cite[Proposition~31]{Martini2001} yields that $[0,1] \ni t \mapsto
    \phi(\gamma(t) - \eta)$ is a~strictly increasing function and we know that
    $\phi(\zeta-\eta) \le R$; thus, we have $\phi(\gamma(t) - \eta) \le R$ for
    all $t \in [0,1]$ and
    \begin{displaymath}
        A \subseteq D \,.
    \end{displaymath}
    Let $w \in P$ and $\omega \in P^*$ be such that $w$ and $\eta$ are linearly
    independent, $\omega(w) = 1$, and $\omega(\eta) = 0$. Define the function
    $f : \R \to \R$ by
    \begin{displaymath}
        f(t) = \omega(\gamma(t)) \quad \text{for $t \in \R$} \,.
    \end{displaymath} Note that $f(1) - f(0) = \omega(\zeta-\xi) = 0$ so $f(1) =
    f(0)$ and, by the mean value theorem, there exists $t_0 \in [0,1]$ such that
    \begin{displaymath}
        0 = f'(t_0) = \omega(\gamma'(t_0)) \,; \quad \text{hence,} \quad
        \gamma'(t_0) = \lambda \eta \quad \text{for some $\lambda \in \R
          \without \{0\}$} \,.
    \end{displaymath} Set $\nu = \gamma(t_0)$. Since $\gamma'(t_0) \in
    \Tan(S,\nu)$ we see that $\eta \in \Tan(S,\nu)$ and $\pi(\nu) \eta = \eta$
    so
    \begin{displaymath}
        \| \pi(\eta) - \pi(\nu) \|_\phi \ge \phi( \pi(\eta) \eta - \pi(\nu) \eta
        ) = \phi( \eta ) = 1
    \end{displaymath} but $\nu \in A \subseteq D$ so this contradicts the choice
    of~$R$.
\end{proof}

\begin{Remark}
    \label{rem:graph}
    Consider the situation as in~\ref{mr:setup} and assume $\phi$ is of
    class~$\cnt{2}$ away from the origin. Let $\varepsilon \in (0,1)$ and
    $\eta \in S$. Set $R = R_{\text{\ref{lem:pi-injective}}}(\varepsilon)$,
    $T = \Tan(S,\eta)$, and $M = S \cap \cballF \phi{\eta}R$. Since
    $\pi(\eta)|M$ is injective and $M$ is compact we see that $\pi(\eta)|M$ is
    a~homeomorphism between~$M$ and~$A = \pi(\eta) \lIm M \rIm \subseteq T$. Set
    \begin{displaymath}
        H = (\pi(\eta)|M)^{-1} \circ \pi(\eta) \quad \text{and} \quad C =
        \pi(\eta)^{-1}\lIm \Int A \rIm \,.
    \end{displaymath} Since $\phi$ is of class~$\cnt{2}$ we see that $M$ is a
    manifold of class~$\cnt{2}$ and $H : C \to M$ is of class~$\cnt{2}$,
    \begin{displaymath}
        H(\zeta) = \xi(\zeta) \quad \text{and} \quad \uD H(\zeta)u = \uD
        \xi(\zeta) u \quad \text{for $\zeta \in S \cap C$ and $u \in
          \Tan(S,\zeta)$} \,.
    \end{displaymath} Differentiating the equation
    \begin{displaymath}
        \uD H(\zeta) \circ \pi(\zeta) u = \uD \xi(\zeta) \circ \pi(\zeta) u
        \quad \text{which holds for $\zeta \in S \cap C$ and $u \in T$}
    \end{displaymath} we get
    \begin{displaymath}
        \uD^2 H(\eta)(u,v) + \uD H(\eta) \bigl( \uD \pi(\eta)uv \bigr) = \uD^2
        \xi(\eta)(u,v) + \pi(\eta) \bigl( \uD \pi(\eta)uv \bigr) \quad \text{for
          $u,v \in T$} \,;
    \end{displaymath} however, if $u,v \in T = \im \pi(\eta)$, then $\uD
    \pi(\eta)uv \in \ker \pi(\eta) = \lin \{\eta\}$ by~\eqref{eq:pi-props} and
    for all $x \in S \cap C$ we also have $\uD H(x)\eta = 0$; hence
    \begin{displaymath}
        \uD^2 H(\eta)(u,v) = \uD^2 \xi(\eta)(u,v) \quad \text{for $u,v \in T$}
        \,.
    \end{displaymath} Since $T$ is tangent at~$\eta \in S$ to the level-set $S$
    of~$\phi$ we have $\uD \phi(\eta) u = 0$ whenever $u \in T$; thus,
    differentiating~\eqref{eq:setup} twice and recalling that $\phi(\eta) = 1$
    and $\xi(\eta) = \eta$ we obtain
    \begin{displaymath}
        \uD^2 H(\eta)(u,v) = \uD^2 \xi(\eta)(u,v) = - \uD^2 \phi(\eta)(u,v) \eta
        \quad \text{for $u,v \in T$} \,.
    \end{displaymath}
\end{Remark}

\begin{Remark}
    In~\ref{thm:lipxi-minkowski} we prove that $\anp{K}{\phi}$ is Lipschitz
    continuous on each of the sets $K_{\lambda,s,t} = \{ x : \arho{K}{\phi}(x)
    \ge \lambda \,, s \le \da{K}{\phi}(x) \le t \}$ defined for $0 < s < t <
    \infty$ and $1 < \lambda < \infty$. Since the proof is a bit technical
    we~briefly describe the main idea. For $x \in K_{\lambda,s,t}$ and $y \in
    \Real{\adim} \without K$ with $\phi(x-a) \le \varepsilon$ we set $a =
    \anp{K}{\phi}(x)$ and choose any $b \in \anp{K}{\phi}(y)$. First we find a
    point $c$ for which $T = \Tan(\Bdry{\cballF{\phi}{x}{\da{K}{\phi}(x)}},a) =
    \Tan(\Bdry{\cballF{\phi}{y}{\da{K}{\phi}(y)}},c)$. For this point we have
    $\phi(a-c) \le 2 \phi(x-y)$; see~\eqref{eq:Fac}. Then we choose $e \in
    \Bdry{\cballF{\phi}{y}{\da{K}{\phi}(y)}}$ and $d \in
    \Bdry{\cballF{\phi}{a+\lambda(x-a)}{\lambda \da{K}{\phi}(x)}}$ which have
    the same orthogonal (with respect to the Euclidean structure induced by
    $\uD^2 \phi(a-x)$) projections onto $T$ as~$a$ and~$b$ respectively; see
    Figure~\ref{fig:lipxi}. We represent
    $\Bdry{\cballF{\phi}{a+\lambda(x-a)}{\lambda \da{K}{\phi}(x)}}$ and
    $\Bdry{\cballF{\phi}{y}{\da{K}{\phi}(y)}}$ locally around~$a$ and~$c$ as
    graphs over~$T$ of functions $g_w$ and $g_u$ of class~$\cnt{2}$
    using~\ref{lem:pi-injective}. Employing~\ref{lem:xi-uni-cont} we can find
    $\varepsilon > 0$ which guarantees that $d$, $e$, and $b$ fit on the graphs
    of~$g_w$, $g_y$, and~$g_y$ respectively. Let $q$ be the signed distance
    from~$T$ such that $q(x-a) > 0$. The crucial point of the proof is in the
    estimates~\eqref{eq:qda-Taylor} and~\eqref{eq:qac-est}, where we use the
    second order Taylor formulas for~$g_w$ and~$g_y$ to compare (both ways) the
    heights $q(d-a)$, $q(e-c)$, and $q(b-c)$ with
    $\lambda^{-1}|\project{T}(d-a)|^2$, $|\project{T}(a-c)|^2$,
    and~$|\project{T}(b-c)|^2$ respectively up to errors expressed in terms of
    the modulus of continuity of~$\uD^2 H$, where $H$ comes
    from~\ref{rem:graph}. Analysing the situation presented on
    Figure~\ref{fig:lipxi} we obtain an estimate of the form
    \begin{displaymath}
        q(b − c) \le q(d-a) + q(e-c) \,,
    \end{displaymath} which, using the comparison mentioned before, is
    translated into
    \begin{displaymath}
        |\project{T}(b-c)|^2 \le \Delta_1 \lambda^{-1} |\project{T}(b-a)|^2 +
        |\Delta_2 |\project{T}(a-c)|^2 \,,
    \end{displaymath} where $\Delta_1$ and $\Delta_2$ can be made arbitrarily
    close to~$1$ by adjusting~$\varepsilon$ depending on the modulus of
    continuity of $\uD^2 H$. This leads to the estimate~\eqref{eq:Tb-a-est} of
    the form
    \begin{displaymath}
        |\project{T}(b-a)|
        \le |\project{T}(b-c)| + |\project{T}(c-a)| \le \Delta_3
        |\project{T}(c-a)| + \lambda^{-1/2} \Delta_4 |\project{T}(b-a)| \,,
    \end{displaymath} where, again, $\Delta_4$ is close to~$1$ given
    $\varepsilon$ is small enough; hence, the last term may be absorbed on the
    left-hand side. Since $|\project{T}(b-a)| \approx |b-a|$ and
    $|\project{T}(c-a)| \approx |x-y|$ we get the conclusion.
\end{Remark}

\begin{Theorem}
    \label{thm:lipxi-minkowski} Consider the situation as
    in~\ref{mr:setup}. Assume
    \begin{gather}
        \text{$\phi|X \without \{0\}$ is of class~$\cnt{2}$} \,, \quad
        K \subseteq X \text{ is closed} \,, \quad
        1 < \lambda < \infty \,, \quad
        x,y \in X \,, \quad
        \arho{K}{\phi}(x) \ge \lambda \,, \\
        a \in \anp{K}{\phi}(x) \,, \quad
        b \in \anp{K}{\phi}(y) \,, \quad
        \eta = \frac{a-x}{\phi(a-x)} \,, \quad
        \uD^2 \phi(\eta)(u,u) > 0 \text{ for $u \in \Tan(S,\eta) \without \{0\}$} \,.
    \end{gather}
    There exist $\varepsilon = \varepsilon(\lambda,\phi,\da{K}{\phi}(x))$ and
    $\Gamma = \Gamma(\lambda,\phi)$ such that
    \begin{displaymath}
        \phi(x-y) \le \varepsilon \quad \text{implies} \quad \phi(a-b) \le
        \Gamma \phi(x-y) \,.
    \end{displaymath}
\end{Theorem}

\begin{proof}
    Clearly we can assume $a \ne b$ and $y \in X \without K$. Define
    \begin{gather}
        r_x = \da{K}{\phi}(x) = \phi(a-x) \,, \quad r_y = \da{K}{\phi}(y)
        = \phi(b-y) \,, \\ c = y + \frac{r_y}{r_x} (a-x) \,, \quad w =
        a+\lambda(x-a) \,, \quad \eta = \frac{a-x}{\phi(a-x)} \,, \quad T =
        \Tan(S,\eta) \,.
    \end{gather}
    Note for the record (see~Figure~\ref{fig:lipxi})
    \begin{displaymath}
        a \in \Bdry{\cballF{\phi}x{r_x}} \cap \Bdry{\cballF{\phi}w{\lambda r_x}}
        \without \cballF{\phi}y{r_y}\,, \quad b,c \in \Bdry{\cballF{\phi}y{r_y}}
        \,, \quad b \notin \cballF{\phi}x{r_x} \,.
    \end{displaymath}
    Recall~\ref{rem:naming-convention} and define
    \begin{displaymath}
        R = R_{\text{\ref{lem:pi-injective}}}(\tfrac 12) \,, \quad H =
        H_{\text{\ref{rem:graph}}}(\tfrac 12,\eta) \,, \quad M =
        M_{\text{\ref{rem:graph}}}(\tfrac 12,\eta) \,, \quad C =
        C_{\text{\ref{rem:graph}}}(\tfrac 12,\eta) \,.
    \end{displaymath}
    Let $q \in X^*$ be such that $q(\eta) = -1$ and $\ker q = T$. Note that
    $\uD^2 \phi(\eta)(\eta,\eta) = 0$ by one-homogeneity of~$\phi$. Let
    $ B : X \times X \to \R $ be the bilinear form such that
    \begin{displaymath}
        \label{eq:B-def}
        B(u,v) = \uD^2 \phi(\eta)(\pi(\eta)u, \pi(\eta)v) + q u \cdot q v
        \quad \text{for $u,v \in X$} \,.
    \end{displaymath}
    By our assumption on $\uD^2 \phi(\eta)$ the map $B$ defines a~scalar product
    on~$X$. In the sequel of this proof we~shall assume the Euclidean structure
    on~$X$ comes from~$B$. In~particular, we~shall use the notations
    \begin{equation}
        \label{eq:Euclidean-str}
        \project{T} = \pi(\eta) \,, \quad u \bullet v = B(u,v) \,, \quad
        \text{and} \quad
        |u| = B(u,u)^{1/2}
        \quad \text{for $u,v \in X$} \,.
    \end{equation}
    \begin{figure}[!htb]
        \centering \includegraphics{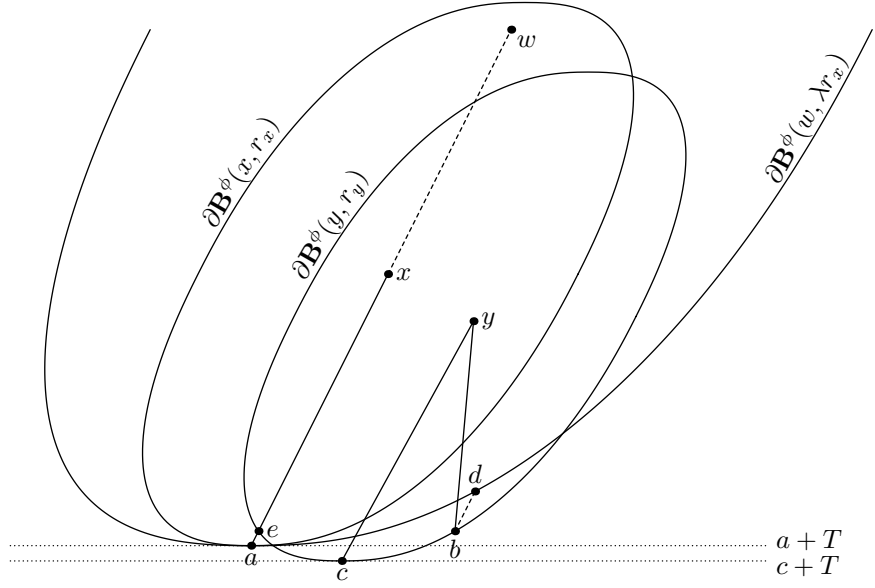}
        \caption{We introduce a Euclidean structure on $X$ so that $x-a$ is
          orthogonal to $T$.}
        \label{fig:lipxi}
    \end{figure}
    Let $\omega_{\lambda}$ be the map obtained from~\ref{lem:xi-uni-cont}. Set
    \begin{gather}
        \Delta_1 = \sup \bigl( \{1\} \cup \bigl\{ |u| : u \in S \bigr\} \bigr)
        \,, \quad \Delta_2 = \min \{ 1 \,, \lambda - 1 \} \,, \\ \sigma(r) =
        \sup \bigl\{ \| \uD^2 H(\zeta) - \uD^2 H(\chi) \| : \zeta,\chi \in C \,,
        |\zeta - \chi| \le r \bigr\} \quad \text{for $0 < r < \infty$} \,,
    \end{gather}
    where the operator norm of the bilinear map
    $\uD^2 H(\zeta) - \uD^2 H(\chi) : X \times X \to X$ is taken with respect to
    the Euclidean structure on~$X$ defined by~\eqref{eq:Euclidean-str}. Choose
    $\varepsilon \in \R$ so that
    \begin{gather}
        \label{eq:eps-choice} 0 < \Delta_1 \varepsilon < 2^{-5} \Delta_2 r_x
        \quad \text{and} \quad \sigma\bigl(4 \Delta_1 (
        \omega_{\lambda}(\varepsilon/r_x) + \varepsilon/r_x ) \bigr) \le 2^{-5}
        \Delta_2 \,.
    \end{gather} Assume $\phi(x-y) \le \varepsilon$. Note that
    \begin{gather}
        \Tan(\Bdry{\cballF{\phi}{y}{r_y}},c) =
        \Tan(\Bdry{\cballF{\phi}{w}{\lambda r_x}},a) =
        \Tan(\Bdry{\cballF{\phi}{x}{r_x}},a) = T \\
        \label{eq:rx-ry} \text{and} \quad \phi(y + (a-x) - c) = |r_y - r_x| <
        \phi(x-y) \le \varepsilon \\ \text{so} \quad
        \label{eq:Fac} \phi(a-c) \le \phi(a-(y + a-x)) + \phi((y + a-x) - c) < 2
        \phi(x-y) \le 2 \varepsilon \,.
    \end{gather} Set $E = \project{T} \lIm C \rIm$ and define
    \begin{gather}
        g_y : \scale{r_y} \lIm E \rIm \to X \quad \text{and} \quad g_w :
        \scale{\lambda r_x} \lIm E \rIm \to X \\ \text{by} \quad g_y(z) = r_y
        H(z/r_y) + y \quad \text{for $z \in \scale{r_y} \lIm E \rIm$} \\
        \text{and} \quad g_w(z) = \lambda r_x H(z/(\lambda r_x)) + w \quad
        \text{for $z \in \scale{\lambda r_x} \lIm E \rIm$}
    \end{gather} so that
    \begin{gather}
        \im g_y = \trans{y} \circ \scale{r_y} \lIm M \rIm \subseteq
        \Bdry{\cballF{\phi}y{r_y}} \,, \quad g_y(0) = c \,, \\ \im g_w =
        \trans{w} \circ \scale{\lambda r_x} \lIm M \rIm \subseteq
        \Bdry{\cballF{\phi}w{\lambda r_x}} \,, \quad g_w(0) = a \,.
    \end{gather} Recall that $H = H \circ \project{T}$ and $a-x,c-y \in \ker
    \project{T} = T^{\perp} = \lin\{\eta\}$. Set
    \begin{displaymath}
        d = g_w(\project{T}(b-a)) \quad \text{and} \quad e =
        g_y(\project{T}(a-c))
    \end{displaymath} and observe that
    \begin{gather}
        \label{eq:badacace} \project{T}(b-a) = \project{T}(d-a) \,, \quad
        \project{T}(a-c) = \project{T}(e-c) \,, \\ z = g_y \circ \project{T}
        (z-c) \quad \text{if $z \in \im g_y$} \,, \quad z = g_w \circ
        \project{T} (z-a) \quad \text{if $z \in \im g_w$} \,, \\
        \label{eq:qba-qda} b \in \Bdry{\cballF{\phi}{y}{r_y}} \without
        \cballF{\phi}{w}{\lambda r_x} \,; \quad \text{hence}\,, \quad q(b-a) <
        q(d-a) \,, \\
        \label{eq:qac-qec} a \in \Bdry{\cballF{\phi}{w}{\lambda r_x}} \without
        \cballF{\phi}{y}{r_y} \,; \quad \text{hence}\,, \quad q(a-c) < q(e-c)
        \,.
    \end{gather} Recalling~\ref{rem:graph} we see that
    \begin{displaymath}
        \uD g_w(0) = \uD H(0) | T = \id{T} \quad \text{and} \quad \uD^2 g_w(0) =
        - \eta (\lambda r_x)^{-1} \uD^2\phi(\eta) \,;
    \end{displaymath} thus, since $|\eta| = 1$ and $q(-\eta) = 1$ the Taylor
    formula~\cite[3.1.11, p.~220]{Federer1969} yields
    \begin{multline}
        \label{eq:qda-Taylor}
        \bigl| q(d-a) - (2 \lambda r_x)^{-1} |\project{T}(d-a)|^2 \bigr|
        \\
        = \bigl| d-a - \project{T} (d-a)
        - (2 \lambda r_x)^{-1} |\project{T}(d-a)|^2 (-\eta) \bigr|
        \\
        = \bigl| g_w(\project{T}(d-a)) - g_w(0)
        - \bigl\langle \project{T}(d-a) \,, \uD g_w(0) \bigr\rangle \tfrac{1}{2}
        - \bigl\langle \project{T}(d-a) \odot \project{T}(d-a) ,\, \uD^2g_w(0)
        - \bigr\rangle
        \bigr|
        \\
        = \bigl| \textstyle{\int_{0}^{1}} (1-s) \bigl\langle \project{T}(d-a)
        \odot \project{T}(d-a) ,\, \uD^2g_w(s \project{T}(d-a)) - \uD^2g_w(0)
        \bigr\rangle \ud \LM^1(s) \bigr|
        \\
        \le (2\lambda r_x)^{-1}
        |\project{T}(d-a)|^2 \sigma
        \bigl( |\project{T}(d-a)| (\lambda r_x)^{-1} \bigr) \,.
    \end{multline} Repeating the above computation twice with $g_y$, $c$, $e$
    and $g_y$, $c$, $b$ in place of $g_w$, $a$, $d$ we get
    \begin{gather}
        \label{eq:qac-est} \bigl| q(e-c) - (2r_y)^{-1} |\project{T}(a-c)|^2
        \bigr| \le (2r_y)^{-1} |\project{T}(a-c)|^2 \sigma \bigl(
        |\project{T}(a-c)| r_y^{-1} \bigr) \\ \text{and} \quad \bigl| q(b-c) -
        (2r_y)^{-1} |\project{T}(b-c)|^2 \bigr| \le (2r_y)^{-1}
        |\project{T}(b-c)|^2 \sigma \bigl( |\project{T}(b-c)| r_y^{-1} \bigr)
        \,.
    \end{gather} Consequently, using~\eqref{eq:badacace}, \eqref{eq:qba-qda},
    and~\eqref{eq:qac-qec}
    \begin{multline}
        \label{eq:qbc-est} 0 < q(b-c) = q(b-a) + q(a-c) \le q(d-a) + q(e-c) \\
        \le \frac{|\project{T}(b-a)|^2}{2\lambda r_x} \bigl( 1 + \sigma \bigl(
        |\project{T}(b-a)| (\lambda r_x)^{-1} \bigr) \bigr) +
        \frac{|\project{T}(a-c)|^2}{2r_y} \bigl( 1 + \sigma \bigl(
        |\project{T}(a-c)| r_y^{-1} \bigr) \bigr)
    \end{multline} and
    \begin{multline}
        |\project{T}(b-c)|^2
        \le \frac{2r_y |q(b-c)|}{1 - \sigma(|\project{T}(b-c)|r_y^{-1})} \\ \le
        \frac{r_y (1 + \sigma(|\project{T}(b-a)|(\lambda r_x)^{-1})) }{\lambda
          r_x (1 - \sigma(|\project{T}(b-c)|r_y^{-1}))} |\project{T}(b-a)|^2 +
        \frac{1 + \sigma(|\project{T}(a-c)|r_y^{-1})}{1 -
          \sigma(|\project{T}(b-c)|r_y^{-1})} |\project{T}(a-c)|^2 \,;
    \end{multline} hence,
    \begin{multline}
        \label{eq:Tb-a-est}
        |\project{T}(b-a)|
        \le |\project{T}(b-c)| + |\project{T}(c-a)| \\ \le \Bigl( 1 + \Bigl(
        \tfrac{1 + \sigma(|\project{T}(a-c)|r_y^{-1})}{1 -
          \sigma(|\project{T}(b-c)|r_y^{-1})} \Bigr)^{1/2} \Bigr)
        |\project{T}(c-a)| + \Bigl( \tfrac{r_y (1 +
          \sigma(|\project{T}(b-a)|(\lambda r_x)^{-1})) }{\lambda r_x (1 -
          \sigma(|\project{T}(b-c)|r_y^{-1}))} \Bigr)^{1/2} |\project{T}(b-a)| \,.
    \end{multline} Recalling \eqref{eq:rx-ry}, \eqref{eq:Fac}, $\phi(x-y) \le
    \varepsilon$, $\arho{K}{\phi}(x) \ge \lambda$ and
    using~\ref{lem:xi-uni-cont} we obtain
    \begin{gather}
        r_x^{-1}|\project{T}(a-b)| \le r_x^{-1} \Delta_1 \phi(a-b) \le \Delta_1
        \omega_{\lambda}(\varepsilon/r_x) \,, \\ r_y^{-1} |\project{T}(a-c)| \le
        r_y^{-1} \Delta_1 \phi(a-c) \le r_x^{-1} 2 \Delta_1 \varepsilon r_x /
        r_y \le r_x^{-1} 4 \Delta_1 \varepsilon \,, \\ r_y^{-1}
        |\project{T}(b-c)| \le r_x^{-1} |\project{T}(b-a)|r_x/r_y + r_y^{-1}
        |\project{T}(a-c)| \le 4 \Delta_1 \bigl(
        \omega_{\lambda}(\varepsilon/r_x) + \varepsilon / r_x \bigr) \,.
    \end{gather} Employing~\eqref{eq:eps-choice}, \eqref{eq:rx-ry}, and noting
    that
    \begin{displaymath}
        \frac{r_y}{r_x} \le 1 + 2^{-5}\Delta_2 \,, \quad \frac{1 + 2^{-5}
          \Delta_2}{1 - 2^{-5} \Delta_2} \le 1 + \frac{\Delta_2}{32} \,, \quad
        \frac{r_y}{r_x} \frac{1 + 2^{-5} \Delta_2}{1 - 2^{-5} \Delta_2} \le 1 +
        \frac{\Delta_2}{2}
    \end{displaymath} we obtain
    \begin{gather}
        \frac{1}{\lambda} \cdot \frac{r_y}{r_x} \cdot \frac{1 +
          \sigma(|\project{T}(b-a)|(\lambda r_x)^{-1})}{1 -
          \sigma(|\project{T}(b-c)|r_y^{-1})} \le \frac{1}{\lambda} \Bigl( 1 +
        \frac{\lambda-1}{2} \Bigr) = \frac{\lambda + 1}{2\lambda} < 1 \,, \\
        \frac{1 + \sigma(|\project{T}(a-c)|r_y^{-1})}{1 -
          \sigma(|\project{T}(b-c)|r_y^{-1})} \le \frac{33}{32} \le 4 \,, \quad
        \Bigl( 1 - \Bigl( \frac{\lambda + 1}{2 \lambda} \Bigr)^{1/2} \Bigr)^{-1}
        = 2\lambda \frac{1 + \bigl( \frac{\lambda + 1}{2 \lambda}
          \bigr)^{1/2}}{\lambda - 1} \le \frac{4 \lambda}{\lambda - 1} \,;
    \end{gather} hence; plugging these estimates to~\eqref{eq:Tb-a-est} yields
    \begin{equation}
        \label{eq:Tba-Tca}
        |\project{T}(b-a)| \le \frac{12\lambda}{\lambda - 1} |\project{T}(c-a)|
        |\,.
    \end{equation} Note that $|\project{T}(a-c)| \le 2 \varepsilon \Delta_1 \le
    \min\{r_x,r_y\}$ by~\eqref{eq:eps-choice} and~\eqref{eq:rx-ry}. In case
    $q(a-c) \ge 0$ we combine~\eqref{eq:Tba-Tca}, \eqref{eq:qbc-est},
    \eqref{eq:qac-est}, \eqref{eq:qac-qec} to get
    \begin{multline}
        \label{eq:qba-final-est}
        |q(b-a)| \le q(b-c) + q(a-c)
        \\ \le |\project{T}(b-a)|^2 \frac{1 + 2^{-5} \Delta_2}{2\lambda r_x} +
        |\project{T}(a-c)|^2 \frac{1 + 2^{-5} \Delta_2}{r_y} \le \Bigl( 2 +
        \frac{ (12\lambda)^2}{(\lambda - 1)^2} \Bigr) |\project{T}(a-c)| \,.
    \end{multline} If $q(a-c) < 0$, then $q(b-a) = q(b-c) + q(c-a) \ge 0$ and we
    get by~\eqref{eq:qba-qda}, \eqref{eq:qda-Taylor}, \eqref{eq:badacace},
    \eqref{eq:Tba-Tca}
    \begin{displaymath}
        |q(b-a)| = q(b-a) \le q(d-a)
        \le |\project{T}(b-a)|^2 \frac{1 + 2^{-5} \Delta_2}{2\lambda r_x} \le
        \frac{ (12\lambda)^2}{\lambda(\lambda - 1)^2} |\project{T}(a-c)| \,.
    \end{displaymath} As a result the final estimate of~\eqref{eq:qba-final-est}
    holds regardless of the sign of $q(a-c)$. Employing~\eqref{eq:Fac}
    \begin{multline}
        |b-a| \le |\project{T}(b-a)| + |q(b-a)|
        \le \Bigl( \tfrac{12\lambda}{\lambda - 1} + 2 + \tfrac{
          (12\lambda)^2}{(\lambda - 1)^2} \Bigr) |\project{T}(c-a)| \\ \le \Bigl(
        \tfrac{12\lambda}{\lambda - 1} + 2 + \tfrac{ (12\lambda)^2}{(\lambda -
          1)^2} \Bigr) \Delta_1 \phi(c-a) \le \Gamma \phi(x-y) \,,
    \end{multline} where $\Gamma = 2 \Delta_1 (\tfrac{12\lambda}{\lambda - 1} +
    2 + \tfrac{ (12\lambda)^2}{(\lambda - 1)^2})$.
\end{proof}

\begin{Corollary}
    \label{cor:xi-lip-cont} Assume $\phi$ is uniformly convex, $K \subseteq X$
    is closed, $0 < s < t < \infty$, $1 < \lambda < \infty$, and
    \begin{displaymath}
        K_{\lambda,s,t} = \bigl\{ x : \arho{K}{\phi}(x) \ge \lambda \,, s \le
        \da{K}{\phi}(x) \le t \bigr\} \,.
    \end{displaymath} Then there exists $\Gamma \in \R$ depending only on $s$,
    $t$, $\lambda$, and~$\phi$ such that
    \begin{displaymath}
        \phi( \anp{K}{\phi}(a) - y ) \le \Gamma \phi(a-b) \quad \text{whenever
          $a \in K_{\lambda,s,t}$, $b \in \Real{\adim}$, $y \in \anp{K}{\phi}(b)$, and
          $\da{K}{\phi}(b) \le t$} \,.
    \end{displaymath} In~particular, $\anp{K}{\phi}|K_{\lambda,s,t}$ is
    Lipschitz continuous.
\end{Corollary}

\begin{proof}
    Assume $a \in K_{\lambda,s,t}$, $b \in \Real{\adim}$, $y \in \anp{K}{\phi}(b)$,
    $x \in \anp{K}{\phi}(a)$, and $\da{K}{\phi}(b) \le t$. Let $\varepsilon =
    \varepsilon_{\ref{thm:lipxi-minkowski}}(\lambda,\phi,s)$. If $\phi(a-b) \le
    \varepsilon$, then the conclusion follows from~\ref{thm:lipxi-minkowski}.
    In case $\phi(a-b) > \varepsilon$, we have
    \begin{displaymath}
        \phi(x - y) \le \phi(x-a) + \phi(a-b) + \phi(b-y) \le \phi(a-b) + 2t \le
        \phi(a-b) ( 1 + 2t/\varepsilon ) \,.  \qedhere
    \end{displaymath}
\end{proof}

\begin{Remark}
    Observe that the bound for the Lipschitz constant of
    $\anp{K}{\phi}|K_{\lambda,s,t}$ obtained in~\ref{thm:lipxi-minkowski}
    explodes with $\lambda \to 1^+$.  This is in accordance
    with~\ref{rem:parabola}.
\end{Remark}

\begin{proof}[Proof of~Theorem \ref{intro_Lipschitz}]
    Since
    \begin{displaymath}
        \grad \da{K}{\phi}(x)
        = \grad \phi(x - \anp{K}{\phi}(x))
        \quad \text{for $x \in \dmn (\grad \da{K}{\phi}) \without K = \Unp{\phi}(K) \without K$}
    \end{displaymath}
    by~\ref{basic_properties_of_da_and_anp}\ref{basic_properties_of_da_and_anp_3}
    we obtain the claim directly from~\ref{cor:xi-lip-cont}.
\end{proof}

\section{Twice differentiability points}
\label{sec:twice-diff}
In this section we prove Theorem~\ref{intro:twice_diff_points}. Recall that
$\ar{K}{\phi}$ was defined
by~\eqref{eq:def-r-phi}, 
pointwise differentiability in~\ref{def:pt-diff}, $\phi$-cut locus
$\Cut^{\phi}(K)$ by~\eqref{eq:def-Cut-locus}, and singular sets
$\Sigma^{\phi}(K)$ and~$\Sigma^{\phi}_2(K)$ in~\eqref{eq:def-Sigma}
and~\eqref{eq:def-Sigma2}.

\begin{Remark}
    \label{remark:inclusions}
    It is well known, and follows from~\ref{single-valuedness_of_anp},
    \ref{basic_properties_of_da_and_anp}\ref{basic_properties_of_da_and_anp_3}
    and~\cite[Theorem~3B]{MR1434446}, that
    \begin{displaymath}
        \Sigma^{\phi}(K) \subseteq \Cut^{\phi}(K) \subseteq \Clos{\Sigma^{\phi}(K)}
        \,.
    \end{displaymath}
\end{Remark}

\begin{Remark}
    \label{rem:parabola}
    Consider the parabola $K = \{ (x,x^2) : x \in \R \}$ with centre of
    curvature at the point $a = (0,\frac 12) \in \R^2$. Then
    $a \in \Cut(K) \cap \Unp{}(K)$. We look at the behaviour of $\anp{K}{}$ on
    the line $\{ (x,\frac 12) : x \in \R \}$. Whenever $0 < x < 8^{-1/2}$,
    setting $b = (2x,\frac 12)$, we have $\anp{K}{}(b) = (\sqrt{x},x)$; hence,
    $\anp{K}{}$ is not differentiable at $a$ and $\da{K}{}$ is not pointwise
    differentiable of order~$2$ at~$a$. Note also that $\anp{K}{}$ is not even
    Lipschitz continuous in any neighbourhood of~$a$. On the other
    hand~\ref{basic_properties_of_da_and_anp}\ref{basic_properties_of_da_and_anp_3}
    yields differentiability of~$\da{K}{}$ at~$a$ (which can also be checked by
    direct computation). We conclude $a \in \Sigma_2(K) \without \Sigma(K)$.
    In~\ref{lem:Cut-subs-Sigma} we prove that this is a~generic situation for
    points in $\Cut(K) \cap \Unp{}(K)$.
\end{Remark} 

\begin{Lemma}
    \label{lem:Cut-subs-Sigma}
    Assume $K \subseteq \R^{\adim}$ is closed, $ x \in \R^{\adim} \without K $,
    and $ \da{K}{\phi} $ is pointwise differentiable of order~$2$ at~$x$.
    
    Then $ \arho{K}{\phi}(x)> 1 $.  In~particular $ \Cut^{\phi}(K) \subseteq
    \Sigma^{\phi}_2(K)$.
\end{Lemma}

\begin{proof}
    Define $r = \da{K}{\phi}(x) $, $ \nu = \an{K}{\phi}(x) $,
    $ a = \anp{K}{\phi}(x) $ and
    $ T = \R^{\adim} \cap \{ v : v \bullet \grad \da{K}{\phi}(x) \}
    $. We~use~\ref{Gariepy-Pepe} to find $ r_1 > 0 $ and a~continuous function
    $ f : T \to T^\perp $ which is pointwise twice differentiable at
    $ \project{T}x $ with $ \Der f(\project{T}x) = 0 $ such that, defining
    $ M = \{\chi + f(\chi) : \chi \in T \} $ and $ U = \oballF{\phi}{x}{r_1}$,
    it~holds $ U \cap S^{\phi}(K,r) = U \cap M $. Decreasing $r_1 > 0$ if necessary, we
    infer from the pointwise twice differentiability of $ f $ in
    $ \project{T}x $ that there exists a polynomial function
    $ P : T \to T^\perp $ of degree at most~$2$ such that
    \begin{equation}
        \label{twice diff and cut locus eq1}
        S^{\phi}(K,r) \cap U \subseteq \R^{\adim} \cap
        \{ y :
        P(\project{T}y) \bullet \grad \da{K}{\phi}(x) \geq y \bullet \grad \da{K}{\phi}(x)
        \} \,.
    \end{equation} Decreasing $r_1 > 0$ even more, we can assume also that $ U
    \without S^{\phi}(K,r) $ is the union of two connected and disjointed open sets $
    U^- $ and $ U^+ $ such that
    \begin{displaymath}
        \{ \chi + P(\chi): \chi \in T \} \cap U \subseteq \Clos(U^+) \cap U \,.
    \end{displaymath} Since $ \oballF{\phi}{a}{r} \cap S^{\phi}(K,r) = \varnothing $ we
    infer $ U \cap \oballF{\phi}{a}{r} \subseteq U^- $. Moreover, it follows
    from \eqref{twice
      diff and cut locus eq1} that there exists $ s > 0 $ such that
    $\oballF{\phi}{x + s \nu}{s} \subseteq U^+ $ (notice $ s < r_1 $) and
    \begin{equation}
        \label{twice diff and cut locus eq2} \cballF{\phi}{x + s \nu}{s} \cap
        S^{\phi}(K,r) = \{ x \} \,.
    \end{equation} Choose $ 0 < \epsilon < \frac{r_1}{4} $. The continuity of $
    \anp{K}{\phi} $ and $ \da{K}{\phi} $ at $ x $ implies that there exists $ 0
    < \delta < \epsilon $ such that $ \phi(b-a) < \epsilon $ and $ \phi(b-y) < r
    + \epsilon $ for every $ b \in \anp{K}{\phi}(y) $ and for every $ y \in
    \oballF{\phi}{x}{2\delta} $. Define $ y = x + \delta \nu $, choose $ b \in
    \anp{K}{\phi}(y) $ and let $ \tau = \sup\{ t : 0 \leq t \leq 1, \; \phi (y +
    t(b-y) - a) > r \} $. Notice
    \begin{gather}
        r - \epsilon \leq \phi(y + \tau(b-y) - a) - \phi(b-a) \leq \phi(y+
        \tau(b-y) -b) = (1-\tau)\phi(b-y) \,, \\ \tau \phi(b-y) \leq \phi(b-y) -
        (r-\epsilon) \leq 2 \epsilon \,, \\ \phi(y + \tau(b-y) - x) \leq
        \phi(y-x) + \tau\phi(b-y) \leq \delta + 2\epsilon < r_1 \,.
    \end{gather} Therefore, $y + \tau(b-y) \in U \cap \cballF{\phi}{a}{r}
    \subseteq \Clos{U^-} \cap U $. Since $ y \in U^+ $ we infer there exists $ 0
    < t \leq \tau $ such that $ y + t (b-y) \in S^{\phi}(K,r) $. Defining $ z = y + t
    (b-y) $ and noting that $ \phi(z-b) = \phi(x-a) $ and $ \phi(y-z) \geq
    \phi(y-x) $ by \eqref{twice diff and cut locus eq2}, we infer
    \begin{displaymath}
        \phi(y-b) = \phi(y - z) + \phi(z -b) \geq \phi\big(y - x\big) +
        \phi(x-a) = \phi(y-a) \,,
    \end{displaymath} whence we conclude that $ a \in \anp{K}{\phi}(y) $ and
    consequently $ \arho{K}{\phi}(x) > 1 $.
\end{proof}

\begin{Remark}
    \label{Cut_Locus}
    Since $ \Leb{\adim}(\Sigma^{\phi}_2(K)) =0 $ by the Alexandrov
    theorem~\cite{MR0003051}, it follows that
    \begin{equation*}
        \Leb{\adim}(\Cut^\phi(K)) =0.
    \end{equation*}
    In a Riemannian setting a conclusion analogous to Lemma
    \ref{lem:Cut-subs-Sigma} is contained in \cite{MR3373591}. A proof
    of~$ \Leb{\adim}(\Cut^\phi(K)) =0$ along different lines can be found in the proof
    of~\cite[Theorem 5.9, Claim 1]{DRKS2020ARMA}, see
    also~\cite[Remark~5.10]{DRKS2020ARMA}.
\end{Remark}

In the next result the classical notion of approximate lower limit of a function
plays a central role. Let us first recall this definition.

\begin{Definition}[\protect{cf.\ \cite[2.9.12]{Federer1969}}]
    \label{def:ap-lim-inf}
    Let $\rho : \R^{\adim} \to \R$ be a  function. The \emph{approximate lower limit} of $ \rho $ at $ x $ is defined as 
    \begin{displaymath}
        \ap \liminf_{y \to x} \rho(y) = \sup \R \cap \bigl\{ t :
        \density^{\adim}(\Leb{\adim} \restrict \{ y : \rho(y) < t \},x) = 0
        \bigr\} \,.
    \end{displaymath}
 \end{Definition}

\begin{Remark}\label{remark approximate lower limit}
    If $ \rho $ is $ \mathscr{L}^n $-measurable, then  $\ap \liminf_{y \to x} \rho(y) \ge \sigma \in \R$ if and only if
    \begin{displaymath}
        \density^{\adim}(\Leb{\adim} \restrict \{ y : \rho(y) \geq t  \},x) = 1 \quad \text{whenever
            $-\infty < t < \sigma$} \,.
    \end{displaymath}
\end{Remark}

The approximate lower limit of an arbitrary function always defines a Borel
function.  This fact can be proved using an argument similar to those of
\cite[Lemma 5.1]{MR3978264}.

\begin{Lemma}\label{lem: Borel regularity}
    Suppose $ f : \mathbf{R}^n \rightarrow \overline{\mathbf{R}} $ is an
    arbitrary function and let $ \underline{f} : \mathbf{R}^n \rightarrow
    \overline{\mathbf{R}}$ be defined as
    \begin{equation*}
        \underline{f}(x) = \ap\liminf_{y \to x} f(y) \qquad \textrm{for $ x \in \mathbf{R}^n $.}
    \end{equation*}
    Then $ \underline{f} $ is a Borel function.
\end{Lemma}

\begin{proof}
    For every $ t \in \mathbf{R} $ we define $ F_t = \{x : f(x) < t\} $ and we set 
    \begin{equation*}
        W_{t, i, r} = \{  y : \mathscr{L}^n(\mathbf{U}(y,r) \cap F_t) \leq i^{-1} r^n\}
    \end{equation*}
    for $ t \in \mathbf{R} $, $ i \in \mathbf{Z}_+ $ and $ r > 0 $. Then we
    prove that the set $ W_{t, i, r} $ is a closed subset of $ \mathbf{R}^n $
    for every $ t \in \mathbf{R} $, $ i \in \mathbf{Z}_+ $ and $ r > 0
    $. Choose a sequence $ y_k \in W_{t, i, r} $ that converges to $ y \in
    \mathbf{R}^n $. Noting that
    \begin{equation*}
        \mathbf{U}(y,r) \subseteq \bigcup_{k=1}\bigcap_{h=k}^\infty \mathbf{U}(y_h, r)
    \end{equation*}
    we conclude from \cite[2.1.5]{Federer1969}
    \begin{equation*}
        \mathscr{L}^n(F_t \cap \mathbf{U}(y,r))
        \leq  \lim_{k\to\infty}\mathscr{L}^n\big( \bigcap_{h=k}^\infty F_t \cap \mathbf{U}(y_h, r)\big)
        \leq \limsup_{k \to \infty}\mathscr{L}^n(F_t \cap \mathbf{U}(y_k,r))
        \leq i^{-1}r^n
    \end{equation*}
    and $ y \in W_{t, i, r} $. Fix now $ \sigma \in \mathbf{R} $, an increasing
    sequence $t_j$ converging to $ \sigma $, and a countable dense subset $ D $
    of $ \mathbf{R} $. Noting that
    \begin{equation*}
        \{x : \underline{f}(x) \geq \sigma\}
        = \bigcap_{j=1}^\infty \bigcap_{i=1}^\infty \bigcup_{k=1}^\infty
        \bigcap \bigg\{ W_{t_j, i, r}: r \in D, \, 0 < r < \frac{1}{k} \bigg\} \,,
    \end{equation*}
    we conclude that $ \{x : \underline{f}(x) \geq \sigma \} $ is a Borel subset of $
    \mathbf{R}^n $; hence, $ \underline{f} $ is a Borel function.
\end{proof}

We consider now the approximate lower envelope of $ \arho{K}{\phi} $ (see
Definition~\ref{anisotropic_rho}).
\begin{Definition}
    \label{def:approximate_lower_envelopes}
    For a closed set $ K \subseteq \mathbf{R}^n $ we define the function $
    \arhol{K}{\phi} : \mathbf{R}^n \rightarrow \mathbf{R} $ as
    \begin{displaymath}
        \arhol{K}{\phi}(x) = \ap \liminf_{y \to x }\arho{K}{\phi}(y)
        \quad \textrm{for $ x \in \mathbf{R}^n $.}
    \end{displaymath}
\end{Definition}

\begin{Remark}
    \label{rem:arhol_le_arho}
    Clearly $ 1 \leq \arhol{K}{\phi}(x) \leq \limsup_{y \to x} \arho{K}{\phi}(y)
    \leq \arho{K}{\phi}(x)$ for $ x \in \Real{\adim}$ by Lemma \ref{rem:rho-usc}. 
\end{Remark}

\begin{Remark}\label{rem: arl less than ar}
    Let $(a, \eta)\in N^\phi(K) $. If there exists $ 0 < r < \ar{K}{\phi}(a,
    \eta) $ and $ \sigma > 1 $ with $ \arhol{K}{\phi}(a+r\eta) \geq \sigma $,
    then it follows from Remarks \ref{remark approximate lower limit} and
    \ref{rem:arhol_le_arho} and Lemma \ref{rem:rKphi-usc} that
    \begin{equation*}
        \ar{K}{\phi}(a, \eta) =  r \arho{K}{\phi}(a+r\eta) \geq r \arhol{K}{\phi}(a +r \eta) \geq \sigma r.
    \end{equation*}
    Therefore it follows from the definition in \eqref{eq: arl} that $
    \arl{K}{\phi}(a, \eta) \leq \ar{K}{\phi}(a, \eta) $ for every $(a, \eta)\in
    N^\phi(K) $.
\end{Remark} 

We recall the following notation: if $ K \subseteq \mathbf{R}^n $ is a closed
set and $ \sigma \geq 1 $ we set
\begin{equation*}
    K_\sigma = \{ x : \arho{K}{\phi}(x) \geq \sigma   \} \sim K.
\end{equation*}
If $ \sigma > 1 $ then $ K_\sigma \subseteq \mathbf{R}^n \setminus (K \cup
\Cut^\phi(K)) $ and $ K_1 = \mathbf{R}^n \sim K $. Compare the next Lemma with Lemma \ref{rem:rho-usc}.

\begin{Lemma}\label{lem: density points}
    For every closed set $ K \subseteq \mathbf{R}^n $ the function $
    \arhol{K}{\phi} $ is a Borel function and it satisfies
    \begin{equation*}
        \arhol{K}{\phi}(x) = t \arhol{K}{\phi}\big(\anp{K}{\phi}(x) + t (x-\anp{K}{\phi}(x))\big)
    \end{equation*}
    for $ x \in \mathbf{R}^{n+1} \sim K $ with $ \arhol{K}{\phi}(x) > 1 $ and $
    0 < t < \arhol{K}{\phi}(x) $.
\end{Lemma}

\begin{proof}
    The function $ \arhol{K}{\phi} $ is a Borel function by Lemma \ref{lem:
      Borel regularity}.
    
    Let $ h_t $  be defined as in Lemma \ref{lem: dilation homeomorphism} for all $ t \in \mathbf{R} $. Suppose $ x \in \mathbf{R}^n \sim K $, $\sigma = \arhol{K}{\phi}(x)  > 1 $ and $ 0 < t < \sigma $. We choose $ 0 < \epsilon < \frac{\da{K}{\phi}(x)}{2} $ and we notice that
    \begin{equation*}
        \mathbf{B}^\phi(x, \epsilon) \subseteq \bigg\{  y   : \frac{1}{2}\da{K}{\phi}(x) \leq \da{K}{\phi}(y) \leq \frac{3}{2}\da{K}{\phi}(x)\bigg\}
    \end{equation*} 
    and, with the help of Lemma \ref{lem: dilation homeomorphism},
    \begin{equation*}
        h_t[K_\sigma \cap \mathbf{B}^\phi(x, \epsilon)]
        \subseteq K_{\sigma/t} \cap
        \bigg\{
        y  : \frac{t}{2}\da{K}{\phi}(x) \leq \da{K}{\phi}(y) \leq \frac{3t}{2}\da{K}{\phi}(x)
        \bigg\} \,.
    \end{equation*}
    Then we infer from Corollary \ref{cor:xi-lip-cont} and Lemma \ref{lem:
      dilation homeomorphism} that $ h_{t}| K_{\sigma} \cap \mathbf{B}^\phi(x,
    \epsilon) $ is a bi-Lipschitz homeomorphism. Since
    $\density^{\adim}(\Leb{\adim} \restrict K_{\sigma},x) =1 $, we employ
    \cite[Theorem 1]{Buczo1992} to conclude that
    \begin{equation}\label{lem: density points eq}
        \density^{\adim}(\Leb{\adim} \restrict K_{\sigma/t},h_t(x)) =1 \quad \textrm{and}\quad t\arhol{K}{\phi}(h_t(x)) \geq \arhol{K}{\phi}(x).
    \end{equation}
    Noting that $ \arhol{K}{\phi}(h_t(x)) \geq \arhol{K}{\phi}(x)/t >
    \sup\{1, 1/t\} $, we can apply the inequality in \eqref{lem: density
      points eq}, with~$ x $ and~$ t $ replaced by $ h_t(x) $ and $ \frac{1}{t}
    $ respectively, to obtain the desired conclusion.
\end{proof}

\begin{Lemma}
    \label{lem:existence_of_real_eigenvalues}
    Suppose
    \begin{gather}
        K \subseteq \Real{\adim} \;\text{is closed} \,, \quad
        x \in \Real{\adim} \without K \,, \quad \arhol{K}{\phi}(x) = \sigma \,, \quad   \arho{K}{\phi}(x) = \lambda, \\
        \text{$ \da{K}{\phi} $ is pointwise differentiable of order $ 2 $ at $x$} \,,
        \\
        T = \R^{\adim} \cap \{ v : v \bullet \grad \da{K}{\phi}(x) = 0 \} \,, 
        \\
        \quad
        h_t(y) = ty + (1-t) \anp{K}{\phi}(y) \quad
        \text{for $y \in \Real{\adim}$ and $t \in \R$} \,.
    \end{gather}
    Then the following statements hold.
    \begin{enumerate}
        
    \item
        \label{existence_of_real_eigenvalues:1} $ \im \Der \an{K}{\phi}(x)
        \subseteq T $.

    \item
        \label{existence_of_real_eigenvalues:2} $ \Der
        \an{K}{\phi}(x)(\an{K}{\phi}(x)) = \Der
        \anp{K}{\phi}(x)(\an{K}{\phi}(x)) =0 $.

    \item
        \label{existence_of_real_eigenvalues:3} There exists a basis $ v_1, \ldots , v_{n-1} $ of $ T $ of eigenvectors of $ \Der \bm{\nu}^\phi_K(x)| T \in \Hom(T,T) $ and the eigenvalues $ \chi_{1} \leq
        \ldots \leq \chi_{\adim-1} $ of $\Der\an{K}{\phi}(x)|T $ are real numbers
        such that
        \begin{displaymath}
            \frac{1}{(1-\lambda)\da{K}{\phi}(x)}\leq \chi_{i} \leq
            \frac{1}{\da{K}{\phi}(x)} \,.
        \end{displaymath}

    \item
        \label{existence_of_real_eigenvalues:4.01}
        $\Der h_t(x)$ is an isomorphism of $\Real{\adim}$ for every $ 0 < t < \lambda $.
        
    \item
        \label{existence_of_real_eigenvalues:4.1}
        If $ \sigma > 1 $ the $ \mathbf{R}^n $-multivalued map $ h_{1/t} $ is
        strongly differentiable at $ h_t(x) $ for every $ 0 < t < \sigma $.
        
    \item
        \label{existence_of_real_eigenvalues:4}
        If $\sigma > 1$, then $ \da{K}{\phi} $ is pointwise differentiable of
        order $2$ at $h_t(x)$ whenever $0 < t < \sigma$.
    \end{enumerate}
\end{Lemma}

\begin{proof}
    Note that $\lambda > 1$ and $1 \le \sigma \le \lambda$ by
    Lemma~\ref{lem:Cut-subs-Sigma} and Remark~\ref{rem:arhol_le_arho} and also
    that $\an{K}{\phi}$ is differentiable at~$x$
    by~\ref{basic_properties_of_da_and_anp}\ref{basic_properties_of_da_and_anp_6}.
    We~choose a function $ \xi : \Real{\adim} \to \Real{\adim} $ such that $
    \xi(y) \in \anp{K}{\phi}(y) $ for $ y \in \Real{\adim} $ and we define
    \begin{gather}
        \nu(y) = \frac{y-\xi(y)}{\da{K}{\phi}(y)} \quad \text{and} \quad \eta(y)
        = \frac{\grad \phi(y-\xi(y))}{|\grad \phi(y-\xi(y))|} \quad \text{for $
          y \in \Real{\adim} \without K $} \,.
    \end{gather} Employing
    \ref{basic_properties_of_da_and_anp}\ref{basic_properties_of_da_and_anp_6}\ref{basic_properties_of_da_and_anp_3}
    we notice that $ \eta $ is differentiable at $ x $,
    \begin{equation}
        \label{eq:etax=grad-da} \eta(x) = \frac{\grad
          \da{K}{\phi}(x)}{|\grad\da{K}{\phi}(x)|} \,, \quad \text{and} \quad \im
        \Der \eta(x) \subseteq T
    \end{equation} since $ |\eta(y)| = 1 $ for every $ y \in \Real{\adim}
    \without K$.  Moreover, we~compute
    \begin{equation}
        \label{eq:existence_of_real_eigenvalues:7}
        \Der \eta(x)u \bullet v
        = \frac{\Der(\grad \phi \circ \nu)(x)u \bullet v}{|\grad \phi(\nu(x))|}
        = \frac{\pt\Der^2\da{K}{\phi}(x)(u,v)}{|\grad \da{K}{\phi}(x)|}
        \quad \text{for $ u,v \in T$} \,;
    \end{equation}
    whence we conclude that $ \Der \eta(x)|T \in \Hom(T,T) $ is
    self-adjoint. Recalling~\ref{rem:phi-phiast} we notice that
    \begin{displaymath}
        \grad \phi^{\ast}(\eta(y)) = \grad \phi^{\ast}(\grad \phi(y-\xi(y))) =
        \grad \phi^{\ast}\Big(\grad
        \phi\Big(\frac{y-\xi(y)}{\da{K}{\phi}(y)}\Big)\Big) = \nu(y)
    \end{displaymath} for $ y \in \Real{\adim} \without K $. Henceforth,
    \begin{equation}
        \label{eq:existence_of_real_eigenvalues:5} \Der \nu(x) = \Der \grad
        \phi^{\ast}(\eta(x)) \circ \Der \eta(x) \,.
    \end{equation} Since $ \Der \grad \phi^{\ast}(v)v = 0$ and $ \Der \grad
    \phi^{\ast}(v) $ is self-adjoint for $ v \in \Real{\adim} \without \{0\} $,
    we conclude that
    \begin{displaymath}
        \Der \grad \phi^{\ast}(v)u \bullet v = u \bullet \Der \grad
        \phi^{\ast}(v)v = 0 \quad \text{for $ u, v \in \Real{\adim} $, $ v \neq
          0 $} \,,
    \end{displaymath} whence we deduce that $ \im \Der \grad \phi^{\ast}(v)
    \subseteq \{ u : u \bullet v =0 \} $ for $ v \neq 0 $, $ \im \Der \nu(x)
    \subseteq T $, and $ \Der \grad \phi^{\ast}(\eta(x))|T \in \Hom(T,T) $ is
    a~positive definite self-adjoint linear map. In~particular,
    we~established~\ref{existence_of_real_eigenvalues:1} and, moreover,
    it~follows from~\eqref{eq:existence_of_real_eigenvalues:5}
    and~\cite[2.25]{DRKS2020ARMA} that the eigenvalues of $ \Der \nu(x)|T $ are
    real numbers.

    To prove \ref{existence_of_real_eigenvalues:2} we notice by
    \ref{single-valuedness_of_anp} that the equations
    \begin{displaymath}
        \xi(\xi(x) + t(x-\xi(x) ) = \xi(x) \quad \text{and} \quad \nu(\xi(x) + t
        (x-\xi(x)) = \nu(x)
    \end{displaymath} hold for $ 0 < t \leq 1 $ and we differentiate with
    respect to $ t $ at $ t = 1 $.

    We now check the estimate claimed in~\ref{existence_of_real_eigenvalues:3}
    for the eigenvalues of $ \Der \an{K}{\phi}(x)|T $. Assume $x=0$ and $\lambda
    = \arho{K}{\phi}(0) > 1$. Define
    \begin{gather}
        r = \da{K}{\phi}(0) \,, \quad W_1 = \oballF{\phi}{\xi(0)}{r} \,, \quad
        W_2 = \oballF{\phi}{ \xi(0) - \lambda \xi(0) }{(\lambda-1) r} \,,
    \end{gather} and observe that $ W_1 \subseteq \{ y : \da{K}{\phi}(y) < r \}
    $, $ W_2 \subseteq \{y : \da{K}{\phi}(y) > r \} $ and $0 \in \Bdry W_1 \cap
    \Bdry W_2 \cap S^{\phi}(K,r) $. Notice that $ \Tan(S^{\phi}(K,r), 0) = \Tan(\Bdry W_1, 0)
    = \Tan(\Bdry W_2, 0) $. Using \ref{Gariepy-Pepe}, we find an open set $ V $
    containing $ 0 $ and three functions $ f : T \to T^\perp $, $ f_1 : T \to
    T^\perp $ and $ f_2 : T \to T^\perp $ such that $ f_1 $ and $ f_2 $ are of
    class~$\cnt{2}$, $ f $ is pointwise differentiable of order $ 2 $ at $ 0 $,
    $ f(0) = f_1(0)= f_2(0)=0 $, $\Der f(0) = \Der f_1(0) = \Der f_2(0) =0 $,
    \begin{gather}
        V \cap S^{\phi}(K,r) = V \cap \{ \chi + f(\chi) : \chi \in T \} \,, \\ V \cap
        \Bdry W_i = V \cap \{ \chi + f_i(\chi) : \chi \in T \} \quad \text{for
          $i \in \{1,2\}$}\,,
    \end{gather} and $ f_1(\chi) \bullet \eta(0) \leq f(\chi)\bullet \eta(0)
    \leq f_2(\chi) \bullet \eta(0)$ for $ \chi \in \project{T}\lIm V \rIm $. In
    particular,
    \begin{displaymath}
        \Der^2 f_1(0) \bullet \eta(0)(u,u) \leq \pt \Der^2f(0) \bullet
        \eta(0)(u,u) \leq \Der^2 f_2(0) \bullet \eta(0)(u,u) \quad \text{for $u
          \in T$}\,.
    \end{displaymath} Let $ \eta_i : \Bdry W_i \to \sphere{\adim-1} $ be the
    exterior unit normal function of $ W_i $ for $i \in \{1,2\}$. Then $
    \eta_1(0) = \eta(0) = - \eta_2(0) $ and we use \ref{Gariepy-Pepe} in
    combination with \eqref{eq:etax=grad-da},
    \eqref{eq:existence_of_real_eigenvalues:7},
    and~\ref{lem:existence_of_real_eigenvalues:aux} to infer
    \begin{multline}
        \label{eq:existence_of_real_eigenvalues:6}
        -\Der \eta_2(0)u \bullet u
        = \Der^2f_2(0)(u,u) \bullet \eta_2(0) \leq \Der \eta(0)u \bullet u \\
        \leq - \Der f_1(0)u \bullet \eta_1(0) = \Der \eta_1(0)u \bullet u \quad
        \text{for $ u \in T $} \,.
    \end{multline} To conclude we use the argument from the third paragraph of
    the proof of \cite[2.34]{DRKS2020ARMA}. First, we find a positive definite
    self-adjoint map $ C \in \Hom(T,T) $ such that $ \Der \grad \phi(\eta(0))|T
    = C \circ C $; then we observe, using \cite[2.33]{DRKS2020ARMA}, that
    \begin{gather}
        C \circ \Der \eta_2(0) \circ C = C^{-1} \circ \Der \grad \phi(\eta(0))
        \circ \Der \eta_2(0) \circ C = (\arho{K}{\phi}(x)-1)^{-1}r^{-1} \id{T}
        \,, \\ C \circ \Der \eta_1(0) \circ C = C^{-1} \circ \Der \grad
        \phi(\eta(0)) \circ \Der \eta_1(0) \circ C = r^{-1} \id{T} \,,
    \end{gather}
    whence we deduce, employing~\eqref{eq:existence_of_real_eigenvalues:6}
    \begin{multline}
        r^{-1}|u|^2 = (C \circ \Der \eta_1(0) \circ C)u \bullet u = (\Der
        \eta_1(0) \circ C)u \bullet C(u) \\ \geq (\Der \eta(0) \circ C)u \bullet
        C(u) \geq -(\Der \eta_2(0) \circ C)u \bullet C(u) \\ = -(C \circ \Der
        \eta_2(0) \circ C)(u) \bullet u = (1-\arho{K}{\phi}(0))^{-1}r^{-1}|u|^2
        \quad \text{for $ u \in T $} \,.
    \end{multline}
    Noting that $ C \circ \Der \eta(0) \circ C $ is a self adjoint map with the
    same eigenvalues as
    $ \Der \nu(0) = \Der \grad \phi(\eta(0)) \circ \Der \eta(0) $ we finally
    obtain the estimate in \ref{existence_of_real_eigenvalues:3}.

    We turn to the proof of~\ref{existence_of_real_eigenvalues:4.01}. Let $ 0 <
    t < \lambda $ and assume $ t \neq 1 $ (notice that $ h_1 =
    \mathbf{I}_{\mathbf{R}^n} $). By contradiction, if there was $v \neq 0$ with
    $\Der h_t(x)v = 0$, then it would follow that
    \begin{displaymath}
        \frac{t}{t-1}v = \Der \anp{K}{\phi}(x)v
    \end{displaymath}
    and, noting that $\an{K}{\phi}(x) \bullet \grad \da{K}{\phi}(x) =
    \phi(\an{K}{\phi}(x)) = 1$ and $ v = w + \kappa \an{K}{\phi}(x) $ for some $
    w \in T $ and $ \kappa \in \Real{} $, we could employ
    \ref{existence_of_real_eigenvalues:2} to compute
    \begin{displaymath}
        \frac{t}{t-1}(\kappa \an{K}{\phi}(x) + w) = \Der \anp{K}{\phi}(x)w = w -
        \da{K}{\phi}(x) \Der \an{K}{\phi}(x)w \,,
    \end{displaymath}
    whence we would deduce from \ref{existence_of_real_eigenvalues:1} that $
    \kappa =0 $, $ w \neq 0 $ and
    \begin{displaymath}
        \Der \an{K}{\phi}(x)w = (\da{K}{\phi}(x)(1-t))^{-1}w \,,
    \end{displaymath}
    which would contradict one of the estimates
    in~\ref{existence_of_real_eigenvalues:3}: in~case $0 < t < 1$, we have
    $(1-t)^{-1} > 1$ and if $1 < t < \lambda$, then $(1-t)^{-1} <
    (1-\lambda)^{-1}$.
    
    Finally, we prove \ref{existence_of_real_eigenvalues:4.1} and
    \ref{existence_of_real_eigenvalues:4}. Let $ \sigma > 1 $ and $ 0 < t <
    \sigma $. We define the $ \mathbf{R}^n $-multivalued maps (see Definition
    \ref{def: restrictions}) $ T $ and $ S $ by
    \begin{equation*}
        T = h_{1/t}|K_{\sigma/t} \qquad \textrm{and} \qquad S = h_t|K_\sigma.
    \end{equation*}
    Then $ T $ is continuous at $ h_t(x) $ and $ S = T^{-1} $ by Lemma \ref{lem:
      density points} (see Definition \ref{def: inverse}). Moreover, since $
    \density^{\adim}(\Leb{\adim} \restrict K_{\sigma},x) = 1 $, it follows from
    Remark \ref{rem:multi-diff} that $ S $ is strongly differentiable at $ x $
    with $ \Der S(x) = \Der h_t(x) $ and it follows from Lemma \ref{lem: density
      points} that
    \begin{equation*}
        \density^{\adim}(\Leb{\adim} \restrict K_{\sigma/t},h_t(x)) =1.
    \end{equation*}
    We apply \cite[Lemma 2]{Zajicek1983diff} to conclude that $ T $ is
    differentiable at $ h_t(x) $ and, noting the Lipschitz property of $ h_{1/t}
    $ over $ K_{\sigma/t} $ that can be deduced from Corollary
    \ref{cor:xi-lip-cont}, we can use Lemma \ref{lem:lip+ap-diff=diff} to
    conclude that $ h_{1/t} $ is strongly differentiable at $ h_t(x)
    $. Therefore $ \anp{K}{\phi} $ is strongly differentiable at $ h_t(x) $ and
    $ \da{K}{\phi} $ is pointwise differentiable of order $ 2 $ at $ h_t(x) $ by
    Lemma~\ref{basic_properties_of_da_and_anp}\ref{basic_properties_of_da_and_anp_6}.
\end{proof}

\begin{Remark}
    The proof of \ref{existence_of_real_eigenvalues:3} shows the following
    fact. \emph{Suppose $ \xi : \mathbf{R}^n \rightarrow \mathbf{R}^n $ is an
      arbitrary function such that $ \xi(y)\in \bm{\xi}^\phi_K(y) $ for $ y \in
      \mathbf{R}^{n} $ and
      \begin{equation*}
          \eta(y)
          = \frac{\grad \phi(y-\xi(y))}{|\grad \phi(y-\xi(y))|} \quad \textrm{for $
            y \in \Real{\adim} \without K $.}
      \end{equation*}
      Then $\Der(\nabla \phi)(\eta(x))|T $ and $  \Der \eta(x)|T $  are self-adjoint maps in $ \Hom(T,T) $, 
      \begin{equation*}
          \Der \bm{\nu}^\phi_K(x)
          = \Der(\nabla \phi)(\eta(x)) \circ \Der \eta(x)
          \quad \textrm{and} \quad
          \Der \eta(x)(u) \bullet v
          = \frac{\pt\Der^2\da{K}{\phi}(x)(u,v)}{|\grad \da{K}{\phi}(x)|}
          \quad \text{for $ u,v \in T$}
      \end{equation*}}
\end{Remark}

The next Lemma clarifies the relation between the function $ \arhol{K}{\phi} $
and the function $ \arl{K}{\phi} $.
\begin{Lemma}\label{lem:approximate_lower_envelopes}
    Suppose $ K \subseteq \mathbf{R}^n $ is closed and $ (a, \eta)\in N^\phi(K) $. Then the following statements are equivalent.
    \begin{enumerate}
    \item $ \arhol{K}{\phi}(a + r \eta) > 1 $ for some $ r > 0 $, 
    \item $ \arl{K}{\phi}(a, \eta) > 0 $, 
    \item $ \arhol{K}{\phi}(a+r\eta) > 1 $ for every $ 0 < r < \arl{K}{\phi}(a, \eta) $.
    \end{enumerate}
    If (a), (b) or (c) holds, then 
    \begin{equation}\label{lem:approximate_lower_envelopes eq}
        \arl{K}{\phi}(a, \eta) =  r\arhol{K}{\phi}(a+r\eta) \qquad \textrm{for every $ 0 < r < \arl{K}{\phi}(a, \eta) $.}
    \end{equation}
\end{Lemma}

\begin{proof}
    It follows from the definition of $ \arl{K}{\phi} $ that (a) implies~(b) and
    it is obvious that (c) implies~(a).  It remains to prove that (b)~implies
    (c) and the equality in \eqref{lem:approximate_lower_envelopes eq}. Let $ 0
    < r < \arl{K}{\phi}(a, \eta) $ and choose $ \sigma > 1 $ and $ 0 < s <
    \ar{K}{\phi}(a, \eta) $ such that
    \begin{equation*}
        r < \sigma s < \arl{K}{\phi}(a, \eta)
        \quad \textrm{and} \quad
        \density^{\adim}(\Leb{\adim} \restrict K_{\sigma},a + s \eta) = 1.
    \end{equation*}
    It follows that $ \arhol{K}{\phi}(a + s \eta) \geq \sigma > 1 $ and employing Lemma \ref{lem: density points} we conclude that 
    \begin{equation*}
        t \arhol{K}{\phi}\big(a +t s\eta\big) = \arhol{K}{\phi}\big(a + s \eta) \qquad \textrm{for $ 0 < t < \sigma $.}
    \end{equation*}
    Since $ \frac{r}{s} < \sigma  $, choosing $ t = \frac{r}{s} $ we obtain  
    \begin{equation}\label{lem:approximate_lower_envelopes eq 1}
        \sigma \leq \arhol{K}{\phi}(a + s\eta) = \frac{r}{s}\arhol{K}{\phi}(a + r\eta)
        \quad \text{and} \quad
        \arhol{K}{\phi}(a + r\eta) \geq \frac{\sigma s}{r} > 1 \,.
    \end{equation}
   Letting $\sigma s $ approach $\arl{K}{\phi}(a, \eta)$, we infer
      from~\eqref{lem:approximate_lower_envelopes eq 1} that $ \arl{K}{\phi}(a,
      \eta) \leq r \arhol{K}{\phi}(a +r \eta) $ for $ 0 < r < \arl{K}{\phi}(a,
      \eta) $.
    Finally if $ 0 < r < \arl{K}{\phi}(a, \eta) $, then $ \arhol{K}{\phi}(a+r\eta)
    > 1 $,
    \begin{equation*}
        \density^{\adim}(\Leb{\adim} \restrict K_{\sigma},a + r \eta) = 1 \,,
        \quad \textrm{and} \quad
        \sigma r < \arl{K}{\phi}(a, \eta)
        \quad \text{for every $ 1 < \sigma < \arhol{K}{\phi}(a + r\eta) $} \,.
    \end{equation*}
    It follows that $ r\arhol{K}{\phi}(a + r\eta)\leq \arl{K}{\phi}(a, \eta) $ and the proof is complete.
\end{proof}

Lemma \ref{lem:approximate_lower_envelopes} implies the Borel measurability of
the reach function $ \arl{K}{\phi} $.
\begin{Lemma}
    For every closed set $ K \subseteq \mathbf{R}^n $ the function $ \arl{K}{\phi} $ is a Borel function.
\end{Lemma}

\begin{proof}
    Let $ \lambda > 0 $ and choose a sequence $r_i \downarrow 0 $ such that $ r_i
    < \lambda $ for every $ i \geq 1 $. It follows from Lemma
    \ref{lem:approximate_lower_envelopes} that
    \begin{equation*}
        \big\{ \arl{K}{\phi} \geq \lambda \big\}
        = {\textstyle \bigcap_{i=1}^\infty} N^\phi(K) \cap
        \big\{
        (a, \eta): r_i\arhol{K}{\phi}(a + r_i \eta) \geq \lambda
        \big\} \,,
    \end{equation*}
    whence we infer from Lemma \ref{lem: density points} that $ \{\arl{K}{\phi}
    \geq \lambda \} $ is a Borel subset of $ N^\phi(K) $.
\end{proof}

\begin{Lemma}
    \label{remark:arl_positive_reach}
    Suppose $ K \subseteq \mathbf{R}^n $ is a closed set, $ \rho > 0 $, $ U =\{
    x : \da{K}{\phi}(x) < \rho \} $, and $ U \cap \Sigma^{\phi}(K) = \varnothing
    $.
    
    Then $\arl{K}{\phi}(a, \eta) \geq \rho $ for every $ (a, \eta)\in N^\phi(K)
    $. In particular, if $ K $ is convex then $ \arl{K}{\phi}(a, \eta) = +
    \infty $ for every $(a, \eta) \in N^\phi(K) $.
\end{Lemma}

\begin{proof}
    First we notice that $ \Clos{\Cut^\phi(K)} \cap U = \varnothing $ by Remark
    \ref{remark:inclusions}. It follows that
    \begin{equation*}
        \ar{K}{\phi}(a, \eta) \geq \da{K}{\phi}(a + \ar{K}{\phi}(a, \eta) \eta ) \geq \rho \qquad \textrm{for every $(a, \eta)\in N^\phi(K) $} 
    \end{equation*}
    and $ N^\phi(K) = \{(\anp{K}{\phi}(x), \an{K}{\phi}(x)) : x \in U\} $.
    Using Lemma \ref{rem:rKphi-usc} we obtain that $ \da{K}{\phi}(x)
    \arho{K}{\phi}(x) \geq \rho $ for every $ x \in U $ and we infer that
    \begin{equation*}
        \arhol{K}{\phi}(x) = \ap \liminf_{y \to x} \arho{K}{\phi}(y) \geq \ap \liminf_{y \to x} \frac{\rho}{\da{K}{\phi}(y)} = \frac{\rho}{\da{K}{\phi}(x)} > 1
    \end{equation*}
    for every $  x\in U $ and the conclusion follows. 
\end{proof}

\begin{Remark}
    \label{example:Q-Cut_not_empty}
    Let $(e_1,e_2)$ be the standard
    basis of~$\R^{2}$ and set $ x_k = 2^{-k} $ for $k \in \natp$. Then 
    \begin{equation*}
        \frac{x_k - x_{k+1}}{x_k} = \frac 12 = \frac{x_{k+1}}{x_k}
        \quad \text{for $ k \in \natp $} \,.
    \end{equation*}
    Define
    \begin{equation*}
        K = \{ x_k e_1 : k \in \natp  \} \cup \{0\} \subseteq \R^2 \,.
    \end{equation*}
    
    Let $ a > 0 $. Evidently $ \arho{K}{}(ae_2) = \infty $ and
    $ (0, e_2) \in N(K) $. We shall prove that $ \arhol{K}{}(ae_2) < \infty $
    and infer that 
    \begin{equation*}
        \ar{K}{}(0, e_2) = \infty > \arl{K}{}(0, e_2) \,. 
    \end{equation*}
    Define $ c_k = \frac{x_k + x_{k+1}}{2} $ to be the centre of the segment
    in~$\R$ joining $ x_{k+1} $ and $ x_k $. Moreover, let $ Q_k $ be the
    $2$-dimensional square with centre in $ c_k e_1 + a e_2 $ and side-length
    $ \frac{x_k - x_{k+1}}{2} = 2^{-k-2} $ and let $ T_k $ be the
    $2$-dimensional triangle with vertices in $ x_{k+1}e_1 $, $ x_k e_1 $, and
    $ c_k e_1 + 2a e_2 $. Clearly, there exists $k_0 \in \natp$ depending only
    on~$a$ such that $ Q_k \subseteq T_k $ for $k \in \natp$ with $k \ge
    k_0$. Noting that
    \begin{equation*}
        \{c_k e_1 : k \geq 1   \} \times \R = \Sigma(K) \,,
    \end{equation*}
    one infers that there is $k_1 \ge k_0$ depending only on~$a$ such that for
    $ k \in \natp$ with $k \ge k_1$ there holds
    \begin{gather}
        \da{K}{}(z) \arho{K}{}(z) \le |c_ke_1 + 2ae_2 - x_ke_1| \le 3a
        \quad \text{for $k \ge k_1$ and $z \in T_k $} \,,
        \\
        \da{K}{}(z) \ge a - \frac{x_k-x_{k+1}}{4} = a - 2^{-k-3}
        \ge \tfrac 12 a
        \quad \text{for $k \ge k_1$ and $z \in Q_k$} \,.
    \end{gather}
    Consequently,
    \begin{displaymath}
        \arho{K}{}(z) \le 6 
        \quad \text{for $k \ge k_1$ and $z \in Q_k$} \,.
    \end{displaymath}
    Now, for $k \ge k_1$ we get
    \begin{multline}
        x_k^{-2}\Leb{2}
        \bigl( \R^2 \cap
        \bigl\{
        z  : |z \bullet e_1| \leq x_k, \; |(z \bullet e_2) - a| \leq  x_k, \; \arho{K}{}(z) \leq 6
        \bigr\}
        \bigr)
        \\
        \geq x_k^{-2}\sum_{h \geq k} \Leb{2}(Q_h)
        = 2^{2k} \sum_{h \geq k} 2^{-2h-4}
        = \frac{1}{12} \,,
    \end{multline} 
    and conclude that
    $ \density^{\ast 2}(\Leb{2} \restrict \{ z : \arho{K}{}(z) \leq 6 \}, ae_2) >
    0 $; hence, $\arhol{K}{}(ae_2) \leq 6$.
    
    Finally, we prove that $ \anp{K}{\phi} $ is not differentiable at
    $ a e_2 $ for each $ a > 0 $. Assume that there exists
    $L = \uD \anp{K}{} \in \Hom(\R^2,\R^2)$. Since $ \anp{K}{}(ae_2) = 0 $ and
    $ \anp{K}{}(ae_2 + x_ke_1) = x_k e_1 $ for $k \in \natp$ it must be
    $L e_1 = e_1$. However,
    $ \anp{K}{}(ae_2 + c_ke_1) = \{x_k e_1, x_{k+1}e_1\} $ for $ k \in \natp$
    and
    \begin{displaymath}
        \frac{| x_k e_1 - \anp{K}{}(ae_2) - L(c_ke_1) |}{c_k}
        = \frac{x_k - c_k}{c_k} = 2^{-k-2} \cdot \tfrac 23 2^{k+1} = \tfrac 13 > 0 \,.
    \end{displaymath}
\end{Remark}

\begin{proof}[Proof of Theorem \ref{intro:twice_diff_points}]
    Notice that the statement in \ref{intro:twice_diff_points:1} is proved in
    Lemma \ref{lem:Cut-subs-Sigma}. The statement in
    \ref{intro:twice_diff_points:propagation} follows combining Lemma
    \ref{lem:approximate_lower_envelopes} with Lemma
    \ref{lem:existence_of_real_eigenvalues}(f).

    Define $\R_{+} = \{ t : 0 < t < \infty \}$ and for $ t \in \R_{+} $ the
    function $f_t : \R^{\adim} \times \R^{\adim} \to \R^{\adim}$ by the
    formula $f_t(a, \eta) = a + t \eta$ whenever $a,\eta \in \R^{\adim}$. Next,
    define $ g: \mathbf{R}^n \sim (K \cup \Cut^\phi(K)) \to N^\phi(K) $ by
    $ g(x)= (\anp{K}{\phi}(x), \an{K}{\phi}(x)) $ for $ x \in
    \mathbf{R}^n \sim (K \cup \Cut^\phi(K))$. Corollary~\ref{cor:xi-lip-cont} yields
    \begin{equation}
        \label{eq:g-Lip-on-K}
        g |K_{\lambda} \cap S^{\phi}(K,t) \text{ is Lipschitz continuous for each $t > 0$ and $\lambda > 1$} \,.
    \end{equation}
    Moreover, for every $ t \in \R_{+} $ it follows from Lemma \ref{rem:rKphi-usc}
    \begin{displaymath}
        N^\phi(K) \cap \{(a, \eta) : t < \ar{K}{\phi}(a, \eta) \}
        = {\textstyle \bigcup_{\lambda > 1} } g \lIm K_{\lambda} \cap S^{\phi}(K,t) \rIm \,.
    \end{displaymath}
      
    We turn to the proof of~\ref{intro:twice_diff_points:2}. Set $ B =
    (\R^{\adim} \without K) \cap \{ x : \arhol{K}{\phi}(x) < \arho{K}{\phi}(x)
    \} $. Since $ \arho{K}{\phi} $ is a~Borel function by~\ref{rem:rho-usc} and
    $\arhol{K}{\phi}$ is $\Leb{\adim}$~measurable by Lemma~\ref{lem: density points} we
    see that $B$ is $\Leb{\adim}$~measurable and it follows from
    \cite[2.9.13]{Federer1969} that $\Leb{\adim}(B) = 0$; hence, the~coarea
    formula~\cite[3.2.11]{Federer1969} yields a~set $J \subseteq \R_{+}$ such
    that $\Leb{1}(\R_{+} \without J) = 0$ and
    \begin{displaymath}
        \Haus{\adim-1}(B \cap S^{\phi}(K,t)) = 0
        \qquad \textrm{for $ t \in J $} \,.
    \end{displaymath}
    For $t \in \R_{+}$ define   
    \begin{displaymath}
        W = N^\phi(K) \cap \{ (a, \eta) : \arl{K}\phi(a, \eta) < \ar{K}{\phi}(a,
        \eta) \} 
        \quad \text{and} \quad
        W_{t} = W \cap \{ (a, \eta): t < \ar{K}{\phi}(a, \eta) \} \,.
    \end{displaymath}
    Since $ f_t\lIm W_t \rIm \subseteq B \cap S^{\phi}(K,t) $ for $ t \in \R_{+} $
    by~Lemma \ref{lem:approximate_lower_envelopes}, it follows that
    \begin{equation}
        \label{intro:twice_diff_points_2 eq2}
        \Haus{\adim-1}(f_t \lIm W_t \rIm) = 0
        \qquad \textrm{for $t \in J$} \,.
    \end{equation}
    Since
    $ g \lIm K_{\lambda} \cap f_t\lIm W_t \rIm \rIm = g \lIm K_{\lambda} \cap S^{\phi}(K,t)
    \rIm \cap W_t $ for $ \lambda > 1 $ and $ t \in \R_{+} $, we~conclude
    from~\eqref{intro:twice_diff_points_2 eq2} and~\eqref{eq:g-Lip-on-K} that
    \begin{displaymath}
        \Haus{\adim-1}(g \lIm K_{\lambda} \cap S^{\phi}(K,t) \rIm \cap W_t) = 0
        \qquad
        \textrm{for each $ \lambda > 1 $ and $t \in J$} \,.
    \end{displaymath}
    Since
    $W_t = \bigcup \{ W_t \cap g \lIm K_{\lambda} \cap S^{\phi}(K,t) \rIm : 1 < \lambda
    \in \Q\}$, it follows that $ \Haus{\adim-1}(W_t) =0 $ for $ \Leb{1} $ almost all
    $ t \in \R_{+} $. Noting that $ W = \bigcup_{t > 0}W_t $ and $ W_t \subseteq W_s $
    for $ s \leq t $, we infer $\Haus{\adim-1}(W) = 0$.

    We finally prove \ref{intro:twice_diff_points:3}. Define
    \begin{displaymath}
        Z = N^{\phi}(K) \cap \bigl\{ (a,\eta) :
        \text{$ a+s\eta \in \Sigma^{\phi}_2(K)$ for all $ 0 < s < \arl{K}{\phi}(a,\eta) $}
        \bigr\} \,.
    \end{displaymath}
    If $ (a, \eta) \in N^\phi(K) \without Z $, then there exists $ 0 < s <
    \arl{K}\phi(a, \eta) $ such that $ \da{K}{\phi} $ is pointwise twice
    differentiable at $ a + s \eta $. Since $\arhol{K}{\phi}(a + s \eta)> 1$ and
    $ \arhol{K}{\phi}(a + s \eta) = s^{-1} \arl{K}\phi(a, \eta) > 1 $ by Lemma
    \ref{lem:approximate_lower_envelopes}, we~infer
    from~\ref{lem:existence_of_real_eigenvalues}\ref{existence_of_real_eigenvalues:4}
    that $ \da{K}{\phi} $ is pointwise twice differentiable at $ a + t \eta $
    for all $ 0 < t < \arl{K}\phi(a, \eta) $.

    Consequently, it remains to show that $ \Haus{\adim-1}(Z) = 0$. This can be
    done with an argument similar as in the proof
    of~\ref{intro:twice_diff_points:2}. Since $ \Leb{\adim}(\Sigma^{\phi}_2(K))
    =0 $ it follows from coarea formula~\cite[3.2.11]{Federer1969} that there is
    $I \subseteq \R_{+}$ such that $\Leb{1}(\R_{+} \without I) = 0$ and
    \begin{equation}
        \Haus{\adim-1}(\Sigma^{\phi}_2(K) \cap S^{\phi}(K,t)) = 0
        \qquad \textrm{for $ t \in I $} \,.
    \end{equation}
    We define $Z_{t} = Z \cap \{ (a, \eta): t < \arl{K}{\phi}(a, \eta) \} $ for
    $ t \in \R_{+} $. Noting that
    $f_t \lIm Z_t \rIm \subseteq \Sigma^{\phi}_2(K) \cap S^{\phi}(K,t)$ for $ t > 0 $,
    we~infer that
    \begin{equation}
        \Haus{\adim-1}(f_t \lIm Z_t \rIm) =0
        \qquad \textrm{for $ t \in I $} \,.
    \end{equation}
    Since
    $ g\lIm K_{\lambda} \cap f_t \lIm Z_t \rIm \rIm = g \lIm K_{\lambda} \cap S^{\phi}(K,t)
    \rIm \cap Z_t $ for $ \lambda > 1 $ and $ t > 0 $, we conclude that
    \begin{displaymath}
        \Haus{\adim-1}(g \lIm K_{\lambda} \cap S^{\phi}(K,t) \rIm \cap Z_t) =0
        \qquad
        \textrm{for $ \lambda > 1 $ and $ t \in I $} \,.
    \end{displaymath}
    It follows that $ \Haus{\adim-1}(Z_t) =0 $ for $ \mathcal{L}^1 $ almost all
    $ t > 0 $. Noting that $ Z = \bigcup_{t > 0}Z_t $ and $ Z_t \subseteq Z_s $
    for $ s \leq t $, we see that $ \Haus{\adim-1}(Z) =0 $.
\end{proof}

\begin{Remark}\label{remark : residualr set}
Recalling the set $ B $ from the proof above and noting  that
\begin{displaymath}
 \{ a + r \eta: (a, \eta)\in N^\phi(K) , \;
    \arl{K}\phi(a, \eta)\leq r < \ar{K}{\phi}(a, \eta) \}
    \subseteq B, 
\end{displaymath}
we conclude that $\mathscr{L}^n(\{a + r\eta : \arl{K}\phi(a, \eta) \leq r <\ar{K}{\phi}(a,
\eta)\}) = 0$. 
\end{Remark}

The following theorem shows that even in the case of convex bodies with $ \mathcal{C}^{1,1} $-boundaries, the dimension of the set $ Z $ in the statement \ref{intro:twice_diff_points:3} can be precisely $ n-1 $. See also \cite[Theorem 1]{MR699027} for a similar construction.

\begin{Theorem}
    \label{convex_example} There exists a convex set $ K \subseteq \R^{2} $ with
    $ \cnt{1,1} $ boundary and a dense set $ S \subseteq \Bdry{K} $ such that
    \begin{enumerate}
    \item
        \label{convex_example:1}
        $ \Haus{1}(S) =0 $ and $ \Haus{s}(S) = + \infty $ for all $ s < 1 $,
    \item
        \label{convex_example:1'}
        $ \Haus{1}(N(K)|S) =0 $ and $ \Haus{s}(N(K)|S) = + \infty $ for all
        $ s < 1 $,
    \item
        \label{convex_example:2}
        $ \anp{K}{}|S^{\phi}(K,r) $ is not differentiable at $ a + r u $ for all
        $ r > 0 $ and $ (a,u) \in N(K)|S $.
    \end{enumerate}
    In particular, both $ \Sigma_2(K) $ and the set of points where $ \da{K}{} $ is
    not directionally differentiable are dense in $ \R^{2} \without K $ with
    Hausdorff dimension $ 2 $.
\end{Theorem}

\begin{proof}
    We choose a dense $ G_\delta $-set $ G \subseteq \R $ with
    $ \Leb{1}(G) = 0 $ and $ \Haus{s}(G) = + \infty $ for all $ s < 1 $. Then
    Zahorski theorem (see~\cite{MR22592} or~\cite{MR2527127}) ensures the
    existence of non-decreasing Lipschitz function $ f : \R \to \R $
    such that $ f $ is not differentiable at each point of $ G $ and it is
    differentiable at each point of $ \R \without G $. Let $ g $ be a primitive
    of $ f $ and notice that $ g $ is a $ \cnt{1,1} $ convex function. Let
    \begin{displaymath}
        K = (\R \times \R) \cap \{ (x, y) : g(x) \leq y \} \,,
        \qquad
        S = (G \times \R) \cap \{ (x, y) : g(x) = y \} \,.
    \end{displaymath}
    Evidently $ K $ is a convex set with $ \cnt{1,1} $ boundary,
    $ \Haus{1}(S)=0 $ and $ \Haus{s}(S) = + \infty $ for every $ s < 1 $.
    If~$\eta : \Bdry{K} \to \sphere{1} $ is the exterior unit normal
    of~$K$ then
    \begin{displaymath}
        \eta(x, f(x)) = \frac{1}{\sqrt{1+ f(x)^2}} \bigl( f(x), -1 \bigr) \qquad
        \textrm{for $ x \in \R $} \,,
    \end{displaymath}
    whence we infer that $ \eta $ is not differentiable at each point of $ S $
    and it is differentiable at all points of $\Bdry{K} \without S$.  Moreover,
    we notice that the map $ \phi : \Bdry{K} \to N(K) $, defined by
    $ \phi(z) = (z, \eta(z)) $ for $ x \in \Bdry{K} $, is a bilipschitz
    homeomorphism and $ N(K)|S = \phi(S) $. Therefore, \eqref{convex_example:1'}
    holds.
    
    To check \ref{convex_example:2} we assume that $ \anp{K}{}|S^{\phi}(K,r) $ is
    differentiable at $ a+ ru $ for some $ r > 0 $ and $ (a,u) \in N(K)|S $.
    We~notice that $ \anp{K}{}| S^{\phi}(K,r) $ is a bilipschitz homeomorphism onto
    $ \Bdry{K} $ with
    \begin{displaymath}
        (\anp{K}{}|S^{\phi}(K,r))^{-1}(b) = b + r \eta(b)
        \quad \text{for $ b \in \Bdry{K} $.}
    \end{displaymath}
    Therefore,
    $ \Der (\anp{K}{}|S^{\phi}(K,r))(a+ ru) : \Tan(S^{\phi}(K,r),a + ru) \to \Tan(\Bdry{K},a)
    $ is a~linear homeomorphism and $(\anp{K}{}|S^{\phi}(K,r))^{-1}$ is differentiable
    at $ a $. This contradicts the fact that $ \eta $ is not differentiable
    at~$a$.
\end{proof}

\subsection*{Declarations}
\paragraph{Funding.} The research of Sławomir Kolasiński was supported by the
National Science Centre Poland grant no.~2016/23/D/ST1/01084 (years 2017-2021).
\paragraph{Conflicts of interest/Competing interests.} On behalf of all authors,
the corresponding author states that there is no conflict of interest.
\paragraph{Availability of data and material.} Not applicable.
\paragraph{Code availability.} Not applicable.

\bigskip

{\small \noindent
  S{\l}awomir Kolasi{\'n}ski \\
  Instytut Matematyki, Uniwersytet Warszawski \\
  ul. Banacha 2, 02-097 Warszawa, Poland \\
  \texttt{s.kolasinski@mimuw.edu.pl}
}

\bigskip

{\small \noindent
  Mario Santilli \\
  Department of Information Engineering, Computer Science and Mathematics,\\
  Università degli Studi dell'Aquila\\
  via Vetoio 1, 67100 L’Aquila, Italy\\
  \texttt{mario.santilli@univaq.it}
}

\end{document}